\documentclass[10pt]{amsart}
\usepackage{amssymb,amsmath,amsfonts,amscd,comment,float,hyperref,mathrsfs, bbm, xcolor}

\newcommand{\R}{\mathbb{R}} 

\newcommand{\rbracket}[1]{\left(#1\right)}
\newcommand{\sbracket}[1]{\left[#1\right]}
\newcommand{\cbracket}[1]{\left\{#1\right\}}
\newcommand{\norm}[1]{\left\|#1\right\|}

\newcommand{\abs}[1]{ \left| #1 \right|}

\newcommand{\bbR}{\mathbb{R}}

\newcommand{\bbT}{\mathbb{T}}

\newcommand{\calC}{ C}

\newcommand{\calF}{\mathcal{F}}

\newcommand{\calO}{\mathcal{O}}
\newcommand{\calP}{\mathcal{P}}

\newcommand{\calR}{\mathcal{R}}

\newcommand{\calV}{\mathcal{V}}

\newcommand{\scrR}{\mathscr{R}}

\newcommand{\scrT}{\mathscr{T}}

\newcommand{\scrV}{\mathscr{V}}

\newcommand{\frR}{\mathfrak R}

\newcommand{\frT}{\mathfrak T}

\DeclareMathOperator{\supp}{supp}

\def\XXint#1#2#3{{%
\setbox0=\hbox{$#1{#2#3}{\int}$}
\vcenter{\hbox{$#2#3$}}\kern-.5\wd0}}

\renewcommand{\leq}{\leqslant}
\renewcommand{\geq}{\geqslant}
\renewcommand{\subset}{\subseteq}

\newcommand{\eps}{\varepsilon}

\theoremstyle{plain}
\newtheorem{theorem}{Theorem}[section]
\newtheorem{proposition}[theorem]{Proposition}
\newtheorem{corollary}[theorem]{Corollary}

\newtheorem{lemma}[theorem]{Lemma}
\theoremstyle{definition}

\theoremstyle{remark}
\newtheorem{remark}[theorem]{Remark}

\newcommand{\NLIalp}{(-\Delta)^{\alpha-1}}

\newcommand{\HP}{\mathbb{R}^2_+}

\newcommand{\diff}[1]{\mathrm{d}{#1}}
\newcommand{\idiff}[1]{\, \diff{#1}}

\newcommand{\pt}{\partial_t}

\usepackage{tikz}

\begin{document}

%

%
\title[]{Self-similar blow-up profile for the one-dimensional reduction of generalized SQG with infinite energy}

\author[T.~Y.~Hou]{Thomas Y. Hou}
\address{Department of Applied Mathematics, California Institute of Technology, Pasadena, CA 91125, USA}
\email{hou@cms.caltech.edu}

\author[X.~Qin]{Xiang Qin}
\address{Department of Applied Mathematics, California Institute of Technology, Pasadena, CA 91125, USA}
\email{xqin2@caltech.edu}

\author[Y.~Sire]{Yannick Sire}
\address{Department of Mathematics, Johns Hopkins University, 3400 N. Charles Street, Baltimore, MD 21218, USA}
\email{ysire1@jhu.edu}

\author[Y.~Wu]{Yantao Wu}
\address{Department of Mathematics, Johns Hopkins University, 3400 N. Charles Street, Baltimore, MD 21218, USA}
\email{ywu212@jhu.edu}


\begin{abstract}
We study the singularity formation mechanisms of the inviscid generalized Surface Quasi-Geostrophic (gSQG) equation on the whole space $\mathbb{R}^2$ and on the upper half-plane $\mathbb{R}^2_+$, allowing infinite energy. In each case, we derive a one-dimensional reduction that captures the leading-order singular behavior of the original 2D system, and use a fixed-point argument to show the existence of finite-time self-similar blow-up solutions for the 1D systems. We also perform numerical simulations for verification and visualization.
\end{abstract}

\maketitle



\section{Introduction}

In this paper we study singularity formation mechanisms in the inviscid generalized surface quasi-geostrophic (gSQG) equation on the whole plane $\mathbb{R}^2$ and on the upper half-plane $\mathbb{R}^2_+$, allowing configurations of infinite energy. The inviscid gSQG family is the active scalar transport system
\begin{equation}\label{eq:gSQG}
\partial_t \theta + u\cdot \nabla \theta = 0, \qquad
u = \nabla^\perp (-\Delta)^{\alpha-1}\theta, \qquad \alpha\in(0,1),
\end{equation}
so that $\alpha=0$ formally corresponds to the 2D Euler vorticity equation and $\alpha=\tfrac12$ is the (inviscid) SQG equation. In the full plane, $u$ is given by a singular integral (Riesz-type) Biot-Savart law; in the half-plane one imposes a boundary condition compatible with \eqref{eq:gSQG} (typically Dirichlet $\theta|_{\partial \mathbb{R}^2_+}=0$ via odd reflection). The fundamental open problem—especially at and beyond the SQG endpoint—is whether smooth solutions can develop a singularity in finite time.

We primarily focus on providing a rigorous, analytically tractable 1D reduced model, derived from the 2D dynamics under a natural ansatz, for both the full-plane and half-plane settings. The central benefit of the one-dimensional reduction is to isolate (and then analyze) a leading-order blow-up mechanism: under a natural boundary-compatible ansatz, the dynamics restricted to the boundary line closes into a one-dimensional model whose velocity gradient is given by a singular integral operator. This reduction couples advection and stretching in a way reminiscent of the Constantin-Lax-Majda  / De Gregorio family of 1D models, where the authors of \cite{HQWW24,CHH21} show that the De Gregorio model develops a self-similar blow-up profile of expanding type. The strategy is to reformulate the search for a self-similar blow-up profile as a fixed-point existence problem for some operator. In this paper, we adopt the same strategy: we use a dynamic rescaling viewpoint to search for self-similar blow-up profiles, and we construct such profiles for the 1D reduced system via a fixed-point argument, distinguishing expanding versus focusing scenarios. We also perform numerical simulations to provide numerical verification and visualization of the self-similar profiles. 

\subsection{Main results}
Our approach is to first derive a closed one-dimensional reduction of the gSQG dynamics within a structured class, and then use dynamic rescaling together with a Schauder fixed-point argument to construct self-similar blow-up profiles. This leads to an expanding-type scenario on $\mathbb{R}^2$ and a boundary-driven focusing-type profile for the half-plane reduction.

For \eqref{eq:gSQG}, we are able to use the one-dimensional reduction to construct finite-time self-similar blow-up solutions for gSQG on $\bbR^2$. It is worth mentioning that the infinite-energy class is a deliberate choice that enables a closed 1D reduction, and we are not claiming generic finite-energy blow-up. 

\begin{theorem}[gSQG on $\bbR^2$] When $\alpha\in(0,1)$, \eqref{eq:gSQG} admits a finite-time self-similar blow-up solution in the form of $\theta(x,y,t) =  - xy f_*\rbracket{\frac{x}{(T-t)^{\widetilde{c_\ell} }}} $ which has infinite energy and the profile $f_*$ and the parameter $ \widetilde{c_\ell}$ satisfy:
\begin{enumerate}
\item $f_*(x)$ is a nonnegative even function and is monotone decreasing on $[0,\infty)$.
\item $f_*(x)$ has compact support.
\item $f_*(\sqrt{x})$ is convex on $[0,\infty)$.
\item $\widetilde{c_\ell}  = - \frac{1}{2-2\alpha}<0$.
\item $f_*$ is smooth in the interior of its support.
\end{enumerate}
\end{theorem}

We also consider the (inviscid) generalized surface quasi-geostrophic (gSQG) on $\HP$:
\begin{equation} \label{eq:gSQGHP} \begin{split} 
    \pt \theta(x,y,t) & + u(x,y,t)\cdot\nabla \theta(x,y,t) = 0 \text{ on } \HP \ , \\
        u_1(x,y,t) & = c_{0,\alpha} \iint\limits_{(\xi,\eta)\in\HP}   \frac{(y-\eta)  \theta(\xi,\eta,t) }{((x-\xi)^2+(y-\eta)^2)^{1+\alpha}}  -  \frac{(y+\eta)  \theta(\xi,\eta,t) }{((x-\xi)^2+(y+\eta)^2)^{1+\alpha}}     \idiff\xi\diff\eta \ , \\
    u_2(x,y,t) & =- c_{0,\alpha}  \iint\limits_{(\xi,\eta)\in\HP}  \frac{(x-\xi)  \theta(\xi,\eta,t) }{((x-\xi)^2+(y-\eta)^2)^{1+\alpha}}  - \frac{(x-\xi)  \theta(\xi,\eta,t) }{((x-\xi)^2+(y+\eta)^2)^{1+\alpha}}  \idiff\xi\diff\eta  \ .
\end{split} \end{equation}
If one formally ignores the dependence of $\theta(x,y,t)$ on $y$ and only looks at the above PDE at the $x$-axis, then we have the following one-dimensional reduction of gSQG on $\HP$:
\begin{equation}\label{eq:1DgSQGHP}  \begin{split}
\partial_t \theta(x,t) + u(x,t) \partial_x \theta(x,t) = 0 \ , \\
u(x,t) = \frac{c_{0,\alpha}}{-2\alpha} \int_{\xi\in\bbR} \frac{\theta(\xi,t)}{|x-\xi|^{2\alpha}}\idiff \xi \ .
\end{split} \end{equation}
The purpose of this one-dimensional reduction is to capture the behavior of \eqref{eq:gSQGHP} near the boundary. Moreover, this reduction allows us to consider a 1D problem which is easier to deal with than the original 2D problem. Based on the same technique, we are able to show that this one-dimensional problem has a finite-time self-similar blow-up solution with a smooth profile.

\begin{theorem}[1D reduction of gSQG on $\HP$] When $\alpha\in(0,\frac12)$, the one-dimensional reduced problem \eqref{eq:1DgSQGHP} of gSQG on $\HP$ admits a finite-time self-similar blow-up solution in the form of $\theta(x,t) =  (T-t)^{\widetilde{c_\theta} - \widetilde{c_\ell} } x f_*\rbracket{\frac{x}{(T-t)^{\widetilde{c_\ell} }}} $ where the profile $f_*$ and the $\alpha$-dependent parameters $ \widetilde{c_\ell} , \widetilde{c_\theta} $ satisfy:
\begin{enumerate}
\item $f_*(x)$ is a positive even function and is monotone decreasing on $[0,\infty)$.
\item $f_*(\sqrt{x})$ is convex on $[0,\infty)$.
\item $\widetilde{c_\theta}>0$, $\widetilde{c_\ell}  >0$, and $\widetilde{c_\theta} - 2\alpha \widetilde{c_\ell}  = -1$.
\item $f_*$ is smooth on $\R$.
\end{enumerate}
\end{theorem}
\subsection{Our contributions}

Our first main contribution in Section \ref{Section: SQGR2} is a rigorous and invertible one-dimensional reduction for inviscid gSQG on the full plane within a natural infinite-energy class. Under a simple structured ansatz, the 2D dynamics close to a 1D transport–stretching model whose velocity gradient is given by an explicit singular integral operator. This reduction is not merely formal: we show that solutions of the reduced 1D system lift to genuine solutions of the original 2D equation in this class. Building on a dynamic-rescaling formulation, we convert the search for finite-time self-similar blow-up into a profile problem and construct an expanding-type self-similar profile. The key analytic innovation is a Schauder fixed-point framework on a carefully tailored function set that encodes monotonicity, a convexity-type “shape control,” and appropriate barriers; these features are designed to make the nonlinear map well-defined, invariant, continuous, and compact in the chosen topology, ultimately yielding a compactly supported profile which is smooth in the interior of its support.

Our second main contribution in Section \ref{Section: SQGHP} is a parallel program for the upper half-plane, where the boundary can create a hyperbolic flow geometry and a distinct (focusing-type) blow-up mechanism. We derive a boundary-based 1D reduced model that captures the leading-order behavior near the boundary and again formulate the self-similar profile system via dynamic rescaling. In contrast to the full-plane setting, the half-plane profile is not compactly supported and exhibits long-range behavior; accordingly, the functional setting and barriers must be redesigned to encode decay and to control tails, while still ensuring the invariance, continuity, and compactness needed for Schauder’s theorem. This produces a smooth focusing-type self-similar profile for the reduced model and provides a mathematically tractable framework—supported by numerics—for understanding boundary-driven singularity formation scenarios in gSQG.

\subsection{Related work}
For the critical dissipative SQG equation, global regularity is known; see, for instance, \cite{Kiselev2006GlobalWF, CaffarelliVasseur10}. For regularity of coupled systems involving subcritical SQG, see \cite{lin2026liquidcrystalstopologicalvorticity, SWZ24}. For the inviscid SQG/gSQG equations, local well-posedness is available in sufficiently regular Sobolev and H\"older settings on $\bbR^2$, as well as in boundary settings and more recent anisotropic H\"older frameworks adapted to the half-plane; representative results include \cite{CMT94, Inci18, Wu05, Constantin_Nguyen_18, CN18, JKY25, CHOI2025113521}. These works also identify sharp thresholds and critical regimes where well-posedness can fail.

A complementary line of research concerns instability, ill-posedness, and low-regularity behavior. Strong ill-posedness, nonexistence, and loss-of-regularity phenomena for SQG/gSQG-type equations have been established in various Sobolev and H\"older classes; see \cite{Inci18, CM24, JK24, CZO24}. At low regularity, weak solutions are classical in many settings (often via vanishing viscosity), see \cite{Resnick95, BGB15, Marchand08}, while nonuniqueness phenomena for related active scalars have been constructed via convex integration in, e.g., \cite{BSV19, CKL21, IM21}. In the patch setting, there are both singularity and regularity results near boundaries, including finite-time singularity scenarios and exclusion criteria for certain splash-type behaviors; see \cite{KRYZ16, GancedoPatel21, JeonZlatos21}. Finally, there are also results constraining self-similar blow-up profiles under integrability/decay assumptions, which are complementary to constructive approaches; see \cite{BRONZI2025266}.

The remainder of the paper is organized as follows. In Section \ref{Section: SQGR2} we derive and analyze the 1D reduction for gSQG on $\bbR^2$, introduce dynamic rescaling, and construct expanding-type self-similar profiles. In Section \ref{Section: SQGHP} we develop the corresponding theory for the half-plane $\bbR^2_+$, again deriving a boundary-based 1D reduction and constructing focusing-type blow-up profiles consistent with numerics. Section \ref{Section: Numerical Simulation} presents numerical simulations for the 1D reduced model and for the gSQG on $\HP$, respectively, and the appendix collects auxiliary analytic estimates used in the fixed-point argument.

\section{Generalized SQG  on $\bbR^2$}\label{Section: SQGR2}
The (inviscid) generalized surface quasi-geostrophic equation (gSQG) is defined as follows on $\bbR^2$:
\begin{equation}\label{Equation: gSQG-R2}
\begin{cases}
    \pt \theta(x,y,t)  + u(x,y,t)\cdot\nabla \theta(x,y,t) = 0   \\
    u(x,y,t) = \nabla^\perp (-\Delta)^{\alpha-1}\theta(x,y,t) 
\end{cases}\tag{gSQG-$\bbR^2$} \ ,
\end{equation}
where the exponent $\alpha\in(0,1)$ and $u=(u_1,u_2)$ is the Riesz transform of $\theta$ and can be explicitly expressed as the following singular integral:
\begin{align*}
  u_1(x,y,t) = \calR_2\theta = c_{0,\alpha} \iint_{(\xi,\eta)\in\bbR^2} \frac{(y-\eta) \theta(\xi,\eta,t)}{((x-\xi)^2+(y-\eta)^2)^{1+\alpha}} \idiff\xi\diff\eta \ , \\
    u_2(x,y,t) =-\calR_1\theta = -c_{0,\alpha} \iint_{(\xi,\eta)\in\bbR^2} \frac{(x-\xi) \theta(\xi,\eta,t)}{((x-\xi)^2+(y-\eta)^2)^{1+\alpha}} \idiff\xi\diff\eta  \ ,
\end{align*}
where $c_{0,\alpha} = \frac{2^{2\alpha-1} \Gamma(1+\alpha)}{\pi \Gamma(1-\alpha)}>0$.
We note that if $\alpha=0$, then \eqref{Equation: gSQG-R2} is just the 2D Euler Equation. If $\alpha=\frac12$, then \eqref{Equation: gSQG-R2} is the SQG Equation.


\subsection{1D reduction for \eqref{Equation: gSQG-R2}} Consider the case that $\theta(x,y,t)$ has the following one-dimensional reduction:
$$ \theta(x,y,t) = y \omega(x,t) \ , $$
then one can compute that \eqref{Equation: gSQG-R2} becomes 
\begin{equation}\label{Equation: gSQG-R2-1Dreduction} \partial_t \omega(x,t) + u_1 \partial_x \omega + \frac{u_2}{y} \omega = 0 \ .\end{equation}
Thus, we have the following explicit expression for $\frac{u_2}{y}$ and $\partial_x u_1$:
\begin{align*}  
 \frac{u_2}{y} = &- c_{0,\alpha} \frac{1}{y} \iint_{(\xi,\eta)\in\bbR^2} \frac{(x-\xi) \eta \omega(\xi,t)}{((x-\xi)^2+(y-\eta)^2)^{1+\alpha}} \idiff\xi\diff\eta   \\
= &- c_{0,\alpha} \frac{1}{y} \iint_{(\xi,\eta)\in\bbR^2} \frac{(x-\xi) (y+\eta) \omega(\xi,t)}{((x-\xi)^2+\eta^2)^{1+\alpha}} \idiff\xi\diff\eta   \\
= &- c_{0,\alpha} \iint_{(\xi,\eta)\in\bbR^2} \frac{(x-\xi) \omega(\xi,t)}{((x-\xi)^2+\eta^2)^{1+\alpha}} \idiff\xi\diff\eta   \\
= &- c_{1,\alpha}\text{P.V.}  \int_{\xi\in\bbR} \frac{(x-\xi) \omega(\xi,t)}{|x-\xi|^{1+2\alpha}} \idiff \xi  \ , 
\end{align*}

\begin{align*}  
\partial_x u_1  = & -2c_{0,\alpha} (\alpha+1) \iint_{(\xi,\eta)\in\bbR^2} \frac{(x-\xi) (y-\eta)\eta \omega(\xi,t)}{((x-\xi)^2+(y-\eta)^2)^{\alpha+2}} \idiff\xi\diff\eta   \\
 = & 2c_{0,\alpha} (\alpha+1) \iint_{(\xi,\eta)\in\bbR^2} \frac{\eta^2 (x-\xi) \omega(\xi,t)}{((x-\xi)^2+\eta^2)^{\alpha+2}} \idiff\xi\diff\eta   \\
 = & c_{1,\alpha} \text{P.V.} \int_{\xi\in\bbR} \frac{(x-\xi)\omega(\xi,t)}{|x-\xi|^{1+2\alpha}} \idiff \xi   \ ,
\end{align*}
where constant $c_{1,\alpha} :=  c_{0,\alpha} \int_{\eta\in\bbR} (1+\eta^2)^{-1-\alpha} \idiff\eta = 2(\alpha+1)c_{0,\alpha} \int_{\eta\in\bbR} \frac{\eta^2}{(1+\eta^2)^{2+\alpha}}\idiff\eta = \frac{2^{2\alpha-1} \Gamma(\frac12+\alpha) }{ \sqrt{\pi} \Gamma(1-\alpha) } $. We note that $\partial_x u_1(x,0,t) = -  \partial_y u_2(x,0,t) = -\lim_{y\to0}u_2/y$ so the reduction naturally gives opposite sign for $u_2/y$ and $\partial_x u_1$.

Thus, let us consider the singular integral operator 
\begin{equation}
\calP_{1,\alpha} \omega(x) := c_{1,\alpha} \text{ P.V.} \int_{\xi\in \bbR} \omega(\xi)\frac{x-\xi}{\abs{x-\xi}^{1+2\alpha}}\, \idiff \xi  \ ,
\end{equation}
which becomes the Hilbert transform when $\alpha=\frac12$. 
 
Thus \eqref{Equation: gSQG-R2-1Dreduction} can be written as
 \begin{equation}\label{Equation: 1D-SQG}
\begin{cases}
    \pt \omega(x,t)  + u \partial_x \omega  = \calP_{1,\alpha}(\omega)\, \omega   \\
   \partial_x u = \calP_{1,\alpha}(\omega)
\end{cases} \ , \tag{1D-gSQG-$\bbR^2$}
\end{equation}
where one can also write the operator $u = \calP_{0,\alpha}(\omega)$ by
\begin{equation}
(\calP_{0,\alpha}\omega)(x) = 
\begin{cases}
 \frac{c_{1,\alpha}}{1-2\alpha}  \int\limits_{\xi \in\bbR} \omega(\xi)|x-\xi|^{1-2\alpha}\, \idiff \xi & \text{ if } \alpha \neq \frac12 
 \\
c_{1,\alpha}  \int\limits_{\xi \in\bbR} \omega(\xi)\log(|x-\xi|)\, \idiff \xi & \text{ if } \alpha = \frac12 
\end{cases}  \ .
\end{equation}

Moreover, this one-dimensional reduction is invertible:  given solution $\omega(x,t), u(x,t)$ of one-dimensional PDE \eqref{Equation: 1D-SQG}, one can define $\theta(x,y,t) := y \omega(x,t)$ and show that this $\theta$ is the solution to \eqref{Equation: gSQG-R2}. To show this, let us define $\psi(x,t) := (-\Delta_{\bbR^1})^{\alpha-1} \omega(x,t)$, and $\Psi(x,y,t) := y \psi(x,t)$. Thus the following computation verifies that $\theta(x,y,t) = (-\Delta)^{1-\alpha}\Psi(x,y,t)$, that is, $\Psi$ is the stream function on $\bbR^2$:\begin{align*} 
& (-\Delta)^{1-\alpha}\Psi(x,y,t) \\
= & \frac{2^{2-2\alpha} \Gamma(2-\alpha)}{\pi | \Gamma(\alpha-1)|}  \iint_{(\xi,\eta)\in\bbR^2} \frac{y\psi(x,t) - \eta \psi(\xi,t)}{\rbracket{(x-\xi)^2+(y-\eta)^2}^{1+(1-\alpha)}} \idiff\xi\diff\eta\\
 = & \frac{2^{2-2\alpha} \Gamma(2-\alpha)}{\pi | \Gamma(\alpha-1)|} \iint_{(\xi,\eta)\in\bbR^2} \frac{(y\psi(x,t) - y \psi(\xi,t)) + (y \psi(\xi,t) - \eta \psi(\xi,t))}{\rbracket{(x-\xi)^2+(y-\eta)^2}^{1+(1-\alpha)}} \idiff\xi\diff\eta\\
 = & \frac{2^{2-2\alpha} \Gamma(2-\alpha)}{\pi | \Gamma(\alpha-1)|}  y \iint_{(\xi,\eta)\in\bbR^2} \frac{\psi(x,t) - \psi(\xi,t)}{\rbracket{(x-\xi)^2+(y-\eta)^2}^{1+(1-\alpha)}} \idiff\xi\diff\eta \\
 = & y \frac{2^{2-2\alpha} \Gamma(2-\alpha)}{\pi | \Gamma(\alpha-1)|} \frac{\sqrt{\pi} \Gamma(\frac32-\alpha)}{\Gamma(2-\alpha)} \text{ P.V.} \int_{\xi\in\bbR} \frac{\psi(x,t) - \psi(\xi,t)}{\abs{x-\xi}^{1+2(1-\alpha)}} \idiff\xi\\
 = & y (-\Delta_{\bbR^1})^{1-\alpha} \psi(x,t) = y \omega(x,t) = \theta(x,y,t) \ .
\end{align*}

Therefore, the solution to the one-dimensional problem \eqref{Equation: 1D-SQG} corresponds to a solution to the two-dimensional problem \eqref{Equation: gSQG-R2}. In view of the scaling property of \eqref{Equation: gSQG-R2} and \eqref{Equation: 1D-SQG}, we are particularly interested in self-similar finite-time blow-up solutions of the form
\begin{equation}
\frac{\theta(x,y,t)}{y} = {\omega(x,t)} = (T-t)^{c_\omega} \Omega\rbracket{ \frac{x}{(T-t)^{c_\ell }}} \ ,
\end{equation}
where $\Omega$ is referred to as the self-similar profile, and $c_\omega, c_\ell $ are the scaling factors. 

Substituting this ansatz into equation \eqref{Equation: 1D-SQG} yields
\begin{align*} 
& -c_\omega (T-t)^{c_\omega-1} \Omega(z) + c_\ell (T-t)^{c_\omega-1}z \Omega'(z) + (T-t)^{2c_\omega + (1-2\alpha)c_\ell } \calP_{0,\alpha}(\Omega)(z) \Omega'(z) \\
 = & (T-t)^{2c_\omega + (1-2\alpha)c_\ell } \calP_{1,\alpha}(\Omega)(z) \Omega(z)  \ , \end{align*}
 where $z=\frac{x}{(T-t)^{c_\ell }}$. Balancing the above equation yields 
 $ (1-2\alpha) c_\ell  + c_\omega + 1 =0 $ and an equation for the self-similar profile
\begin{equation}\label{Equation: SelfSimilarProfileOmega} (c_\ell   x +   \calP_{0,\alpha}(\Omega)(x)) \Omega'(x)  =  (c_\omega + \calP_{1,\alpha}(\Omega)(x)) \Omega(x) \ . \end{equation}
We note that if $(\Omega(x), c_\ell , c_\omega)$ is a solution to the above equation, then so is 
\begin{equation}\label{Equation: scaling invariant}
(\Omega_{\lambda,\gamma}(x), c_{l,\lambda,\gamma}, c_{\omega,\lambda,\gamma}) = (\gamma\Omega(\lambda x), \gamma \lambda^{2\alpha-1} c_\ell ,  \gamma \lambda^{2\alpha-1} c_\omega )  \ .
\end{equation}

This means that we can relax the restriction $c_\omega + (1-2\alpha) c_\ell  = -1  $ to any $c_\omega + (1- 2\alpha)c_\ell  < 0 $. In fact, for any $(c_\omega, c_\ell )$ such that $c_\omega + (1- 2\alpha ) c_\ell  < 0 $, one can define
$$ (\widetilde{c_\omega},\widetilde{c_\ell }) = \rbracket{\frac{c_\omega}{ (2\alpha-1) c_\ell  - c_\omega} , \frac{c_\ell }{(2\alpha-1) c_\ell  - c_\omega} } \ ,$$
which satisfies $\widetilde{c_\omega} +(1 -2\alpha)\widetilde{c_\ell }=-1$.

 Furthermore, we look for solutions that satisfy the following conditions:
 \begin{itemize}
 \item Odd symmetry: $\Omega(x)$ is an odd function of $x$, i.e., $\Omega(-x) = - \Omega(x)$.
 \item Regularity: $\Omega\in H^1_{loc}(\bbR)$.
 \item Non-degeneracy: $\Omega'(0)\neq0$.
 \end{itemize}

In view of \eqref{Equation: scaling invariant}, we can take $\Omega'(0)=-1$ without loss of generality. Thus, let us consider the following change of variable
$$ f(x) :=\frac{\Omega(x)}{x\Omega'(0)} = -\frac{\Omega(x)}{x}, \quad v(x) := c_\ell  x + \calP_{0,\alpha}(\Omega)(x), \quad g(x) := \max\rbracket{0, \frac{v(x)}{xv'(0)}} \ . $$
We note that 
\begin{equation}\begin{split}
v(x)  & =
 \begin{cases}
c_\ell  x +  \frac{-c_{1,\alpha} }{1-2\alpha}\int_{\xi\in\bbR} \xi f(\xi) |x-\xi|^{1-2\alpha}\, \idiff \xi & \text{ when } \alpha\neq\frac12  \\
c_\ell x - c_{1,\alpha} \int_{\xi\in\bbR} \xi f(\xi) \log(|x-\xi|)\, \idiff \xi & \text{ when } \alpha=\frac12 \\
 \end{cases}  \\
 & =
 \begin{cases}
c_\ell  x +  \frac{c_{1,\alpha}}{1-2\alpha}\int_{\xi>0} \xi f(\xi) ( |x+\xi|^{1-2\alpha} -  |x-\xi|^{1-2\alpha}) \, \idiff \xi & \text{ when } \alpha\neq\frac12 \\
c_\ell x + c_{1,\alpha} \int_{\xi>0} \xi f(\xi) \log\rbracket{\frac{|x+\xi|}{|x-\xi|}} \, \idiff \xi & \text{ when } \alpha=\frac12 \\
 \end{cases} \ , \\
  v'(x)  & = c_\ell  -  c_{1,\alpha} \text{ P.V.}\int_{\xi\in\bbR} \xi f(\xi) \frac{x-\xi}{|x-\xi|^{1+2\alpha}} \, \idiff \xi  \ .
\end{split}\end{equation}

Thus \eqref{Equation: SelfSimilarProfileOmega} can be written as 
$$ \frac{f'}{f} =  \frac{g'}{g}, \quad f(0) = g(0)=1, c_\omega = c_\ell  \ .$$
As a consequence, we reduce the equation \eqref{Equation: SelfSimilarProfileOmega} for the self-similar profile to the following form:
\begin{equation}\label{Equation: fixedPointEqn} \begin{split}
f(x) & = \max\rbracket{0, \frac{v(x)}{xv'(0)}} \ , \\
  \frac{v(x)}{x}   & =\begin{cases}
c_\ell   +  \frac{c_{1,\alpha} }{1-2\alpha}\int_{\xi>0} f(\xi)\frac{\xi}{x} ( |x+\xi|^{1-2\alpha} -  |x-\xi|^{1-2\alpha}) \, \idiff \xi & \text{ when } \alpha\neq\frac12 \\
c_\ell  + c_{1,\alpha} \int_{\xi>0} f(\xi)\frac{\xi}{x} \log\rbracket{\frac{|x+\xi|}{|x-\xi|}} \, \idiff \xi & \text{ when } \alpha=\frac12 \end{cases} \ .
\end{split}  \end{equation}


\subsection{Existence of solutions by a fixed-point method}
Our goal of this section is to show that the nonlinear system \eqref{Equation: fixedPointEqn} admits non-trivial solutions. We do so by converting the problem into a fixed-point problem of some nonlinear map. In detail, we are going to define a linear operator $\scrT_\alpha(f)(x) := \frac{v(x)}{x}- v'(0)$ which simplifies \eqref{Equation: fixedPointEqn} to $f(x) = \max(0,1+\frac{\scrT_\alpha(f)(x)}{v'(0)})$, and then the problem becomes finding a fixed-point of some nonlinear map $\scrR_\alpha$ which will be defined afterwards. Notice that we do not want to directly define $\scrR_\alpha(f)$ as $\max(0,1+\frac{\scrT_\alpha(f)(x)}{v'(0)})$ because the value of $c_\omega= c_\ell $  is underdetermined in the definition of $v'(0)$. To make $\scrR_\alpha(f)$ a well-defined map, we instead let $\scrR_\alpha(f):=\max(0,1+\frac{\scrT_\alpha(f)(x)}{c(f)})$ where the functional $c(f)$ depends only on $f$ and is chosen to guarantee the renormalization $\lim_{x\to0+}\frac{f_*'(x)}{2x}=-1$, i.e., $f_*(x)\approx 1-x^2$ when $x\to0+$, for the self-similar profile $f_*$. The full details of one-to-one correspondence between solutions to \eqref{Equation: fixedPointEqn} and fixed-points of $\scrR_\alpha$ will be shown in Proposition \ref{Proposition: relation1Dmodel2DSQGR2}.

To show existence of a fixed-point of the proposed map $\scrR_\alpha(f)$, we are going to use the Schauder fixed-point theorem. The reasons to use Schauder rather than contraction mapping are: primarily, we mostly care about existence and do not care about uniqueness and rate of convergence; secondly,  it is more difficult to show quantitative results on contraction rather than qualitative results such as continuity, compactness, and convexity.

\begin{remark}  This fixed-point strategy follows the approach developed in \cite{HQWW24} for the one-dimensional generalized Constantin–Lax–Majda (gCLM) equation: one introduces a nonlinear operator on a suitably chosen invariant, convex, and compact set of monotone “shape-controlled” functions, verifies that the operator maps this set into itself and is continuous, applies the Schauder fixed-point theorem to obtain a profile, and then bootstraps regularity. In both settings, the monotonicity and convexity-type constraints are carefully constructed so that—after rewriting the nonlocal terms in a form amenable to integration by parts—the relevant operators preserve qualitative shape properties and yield the compactness estimates required by Schauder’s theorem. The precise choice of constraints is, however, kernel-dependent: compared with gCLM, the gSQG reduction involves a different singularity order and far-field behavior, which necessitates working in a different topology (e.g., weighted control and different decay/barrier conditions) to establish invariance, continuity, and compactness for our nonlinear map $\scrR_\alpha$. Accordingly, we select an appropriate Banach space $\calV_1$ tailored to these operator features.
\end{remark}

 To this end, we need to select an appropriate Banach function space $\calV_1$ in which we can establish invariance, continuity, and compactness of our nonlinear map $\scrR_\alpha$.

\subsubsection{Details of function set $\calV_1$ as the invariant set for the fixed-point method.}
Consider a Banach space of continuous even functions,
$$ \calV_0 := \{ f\in C(\bbR)\, : \, f(-x) = f(x), \norm{\rho f}_{L^\infty} < \infty \} \ , $$
endowed with a weighted $L^\infty$-norm $\norm{\rho f}_{L^\infty}$, referred to as the $L^\infty_\rho$-norm, where $\rho(x) = (1+x)^{-\alpha}$. Moreover, we consider a closed (in the $L^\infty_\rho$-norm) and convex subset of $\calV_0$,
$$ \calV_1 :=  \cbracket{ f\in\calV_0 \, : \, 
\begin{split} 
& f(0)=1,f \text{ is nonnegative and non-increasing on } [0,\infty), \\
& f(\sqrt{x}) \text{ is convex in } x, f\geq \max(0,1-x^2), f'_-(1/2)\leq -\eta
\end{split} }  \ , $$
where 
$$ \eta = 2\rbracket{ \frac{3}{2(3-2\alpha)(5-2\alpha)(1+4^\alpha)} }^{\frac1{\min(\alpha,1-\alpha)}} \ .$$
The function set $\calV_1$ will act as the invariant set for our fixed-point method. Here and below, we use $f_-'$ and $f_+'$ to denote the left and the right derivatives of a function $f$. The condition $f_-'(1/2) \leq -\eta$ with $\eta>0$ is to rule out the constant function from the  set $\calV_1$. In fact, $f\in\calV_1$ implies that 
\begin{equation}\label{Equation: fUpperLowerBound}  \max(0,1-x^2) \leq f(x) \leq \max(1-\eta x^2, 1-\eta/4) \ ,\end{equation}
where the upper bound follows from the assumptions that $f(\sqrt{x})$ is convex in $x$, $f_-'(1/2)\leq -\eta$, and $f(x)$ is non-increasing on $[0,\infty)$. On the other hand, for $\calV_1$ to be an invariant set in our fixed-point method, $\eta$ needs to be sufficiently small. The particular choice of $\eta$ here is actually determined through a bootstrap argument that will be clear later.

We remark that, though a function $f\in\calV_1$ is not required to be differentiable, the one-sided derivatives $f_-'(x)$ and $f_+'(x)$ are both well defined at every point $x$ by the convexity of $f(\sqrt{x})$ in $x$. In what follows, we will abuse notation and simply use $f'(x)$ for $f'_-(x)$ and $f_+'(x)$ in both weak sense and strong sense. For example, when we write $f'(x)\leq C$, we mean both  $f'_-(x) \leq C$ and $f'_+(x) \leq C$ at the same time. In this context, the non-increasing property of $f$ on $[0,\infty)$ can be represented as $f'\leq0$ when $x\geq0$.

The reasons to choose those constraints in the function set $\calV_1$ are as follows. The assumption $f(0)=1$ is chosen by scaling invariance in \eqref{Equation: scaling invariant}. $f$ being nonnegative is a natural assumption in finding a simple formulation of the self-similar profile. It is also possible to search for self-similar profiles that change sign, but that would be more difficult to analyze. The monotonicity of $f$ ensures the kernel acts with a definite sign, prevents oscillations, makes many estimates easy. The square-root convexity, or convexity after reparameterization, is the key structural condition that makes the integral operator preserve shape, and it provides control of one-sided derivatives even without smoothness, which allows us to bootstrap regularity later. The properties of monotonicity of $f$ and convexity of $f(\sqrt{x})$ are also preserved by the map $\scrR_\alpha$, see proofs in Proposition \ref{Proposition:RaV1R2}. The assumption $f_-'(1/2)\leq -\eta$ is to give an upperbound of $f$ that rules out the constant-valued function $f\equiv1$. Moreover, this assumption, together with the lowerbound $f\geq\max(0,1-x^2)$, guarantees the normalization constant $c(f)$ in the map $\scrR_\alpha(f)$ does not degenerate. Finally, the weight $\rho(x) = (1+x)^{-\alpha}$ makes $\calV_1$ compact on an unbounded domain. 

One extra remark to justify the lowerbound. The main goal is to constrain the first- and second-order behavior of $f$ near the origin, i.e. $f(x)-1=O(x^2)$ as $x\to 0$. Moreover, we are searching for an expanding-type self-similar blow-up, which often corresponds to a compactly supported profile, and it is therefore natural to choose a compactly supported lowerbound.


\subsubsection{ Introducing maps $\scrT_\alpha(f)$ and $\scrR_\alpha(f)$. }
To simplify the expression for $\frac{v(x)}{x}$ in \eqref{Equation: fixedPointEqn}, we will separate some terms in the singular integral. We notice that when $\alpha\neq\frac12$, we can let $\gamma = 1-2\alpha$ and compute that 
\begin{align*}
& \frac{1}{\gamma} \int_{\xi>0} f(\xi) \sbracket{\frac{\xi}{x}( |x+\xi|^{\gamma} -  |x-\xi|^{\gamma}) -2\gamma \xi^{\gamma} } \idiff \xi\\
= &\frac{1}{\gamma}   \int_{\xi>0} f(\xi)\partial_{\xi} \Bigg[ \frac{(1+\gamma)\xi^2-x^2}{(1+\gamma)(2+\gamma)x}(|\xi+x|^\gamma - |\xi-x|^\gamma) \\
& \quad\quad\quad\quad\quad\quad\quad + \frac{\gamma \xi}{(1+\gamma)(2+\gamma)}(|\xi-x|^\gamma + |\xi+x|^\gamma) - \frac{2\gamma}{1+\gamma}\xi^{\gamma+2} \Bigg]  \idiff \xi \\
= & -\frac{1}{\gamma} \int_{\xi>0}f'(\xi)  \Bigg[\frac{(1+\gamma)\xi^2-x^2}{(1+\gamma)(2+\gamma)x}(|\xi+x|^\gamma - |\xi-x|^\gamma) \\
& \quad\quad\quad\quad\quad\quad\quad + \frac{\gamma \xi}{(1+\gamma)(2+\gamma)}(|\xi-x|^\gamma + |\xi+x|^\gamma) - \frac{2\gamma}{1+\gamma}\xi^{\gamma+2}  \Bigg] \idiff \xi\\
= & \int_{\xi>0} f'(\xi) \xi^{\gamma+1} F_{1,\gamma}(x/\xi) \, \idiff \xi \ .
\end{align*}
Meanwhile, when $\alpha=\frac12$, one can compute that 
 \begin{align*}
& \int_{\xi>0} f(\xi) \sbracket{ \frac{\xi}{x} \log\rbracket{\frac{|x+\xi|}{|x-\xi|}}-2 }\, \idiff \xi \\
= & \int_{\xi>0} f(\xi) \partial_{\xi}\sbracket{  -\xi + \frac{\xi^2-x^2}{2x}\log\abs{\frac{x+\xi}{x-\xi}} }\, \idiff \xi \\
= &  \int_{\xi>0} f'(\xi) \xi F_{1,0}(x/\xi)\, \idiff \xi \ ,
\end{align*}
where the definition of auxiliary functions $F_{1,\gamma}$ can be seen in Appendix \ref{Section: Auxiliary functions}. Thus we define $\alpha$-dependent nonlinear map 
\begin{equation}\label{Equation: OperatorTalpha} \scrT_\alpha(f)(x) = c_{1,\alpha} \int_{\xi>0} f'(\xi) \xi^{2-2\alpha} F_{1,1-2\alpha}(x/\xi)\, \idiff \xi \ ,\end{equation}
which gives the following equivalent expression for \eqref{Equation: fixedPointEqn}:
$$ f(x) = \max\rbracket{ 0, 1 + \frac{\scrT_\alpha(f)(x)}{v'(0)}  } \ .$$

This definition only uses the integral on $[0,\infty)$ since $f\in\calV_1\subseteq\calV_0$  is an even function. We will always employ this symmetry property in the sequel. One should note that $\scrT_\alpha$ is well-defined for any $f\in\calV_1$, since for each fixed $x$, the kernel of $\scrT_\alpha$ decays like $\xi^{-1-2\alpha}$ as $\xi\to\infty$.

From definition of $F_{1,1-2\alpha}$ in Lemma \ref{Lemma: AuxiliaryFunctionF1}, one can show that   $F_{1,1-2\alpha}'(1/t) = t^{1+2\alpha}F_{1,1-2\alpha}'(t)$, $F_{1,1-2\alpha}(0) = 0, F_{1,1-2\alpha}'(0) = 0$, $\lim_{x\to+\infty} F_{1,1-2\alpha}(x) =   \frac{1}{1-\alpha}$,  $F_{1,1-2\alpha}'(t) > 0$ on $t\in(0,\infty)$.

Then we know that 
\begin{equation*}
\begin{split}
& \scrT_\alpha(f)(0) = 0 \ , \\
&  - \lim_{x\to\infty} \scrT_\alpha(f)(x)= 2 c_{1,\alpha} \int_{\xi>0} (f(\xi)-f(+\infty))\xi^{1-2\alpha} \idiff \xi :=b(f)\ , \\
& \frac{v(x)}{x} = \scrT_\alpha(f)(x) + v'(0)  \ .
\end{split}
\end{equation*}

Let us define operator
\begin{equation}\label{Equation: OperatorRalpha} \scrR_\alpha(f)(x) := \max\rbracket{0, 1 + \frac{\scrT_\alpha(f)(x)}{c(f)} } \ , \end{equation}
where $$ c(f) :=  -\frac{1+2\alpha}{3} c_{1,\alpha} \int_{\xi>0} \xi^{-2\alpha}f'(\xi)\, \idiff \xi  = \frac{2\alpha(1+2\alpha)}{3} c_{1,\alpha} \int_{\xi>0} \frac{1-f(\xi)}{\xi^{1+2\alpha}}  \, \idiff \xi  \ . $$

Since $\scrT_\alpha(f)(0)=0$, we have $\scrR_\alpha(f)(0)=1$ in all cases. The ratio $b(f)/c(f)$ will be an important value that occurs frequently in what follows, as it determines the asymptotic behavior of $\scrR_\alpha(f)$:
\begin{equation}\label{Equation: AsymptoticBehavior}
\scrR_\alpha(+\infty) := \lim_{x\to\infty} \scrR_\alpha(f)(x) = \max\rbracket{0,1-\frac{b(f)}{c(f)}}.
\end{equation}
We note that $c(f)$ must be strictly positive and finite for any $f\in\calV_1$. Actually, in view of \eqref{Equation: fUpperLowerBound}, we have 
$$ c(f) = \frac{2\alpha(1+2\alpha)}{3} c_{1,\alpha} \int_{\xi>0} \frac{1-f(\xi)}{\xi^{1+2\alpha}}  \, \idiff \xi \geq  \frac{2^{4\alpha-2}\Gamma(\frac32+\alpha)}{3\sqrt{\pi}\Gamma(2-\alpha)}\eta \ .$$
The lower bound of $c(f)$ explains why we need to impose the condition $f'(\frac12)\leq -\eta$ on $\calV_1$: to make sure $c(f)$ is strictly positive so that $\scrR_\alpha(f):=\max\rbracket{0, 1 + \frac{\scrT_\alpha(f)(x)}{c(f)} }$ is finite for any $f\in\calV_1$.

We now aim to study the fixed-point problem $f=\scrR_\alpha(f), f\in\calV_1$. As the core idea of this paper, the following proposition explains how a fixed-point of $\scrR_\alpha$ is related to a solution to \eqref{Equation: SelfSimilarProfileOmega}.

\begin{proposition}\label{Proposition: relation1Dmodel2DSQGR2} For any $\alpha\in(0,1)$, if $f\in\calV_1$ is a fixed-point of $\scrR_\alpha$, i.e., $\scrR_\alpha(f)=f$, then $\lim_{x\to+\infty} f(x)=0$, $b(f)<\infty$, and $\Omega=-xf$ is a solution to \eqref{Equation: SelfSimilarProfileOmega} with
\begin{equation}\label{Equation: c_l} c_\omega = c_\ell  = c(f) - b(f) = \frac{2\alpha(1+2\alpha)}{3} \int_{\xi>0} \frac{1-f(\xi)}{\xi^{1+2\alpha}}  \, \idiff \xi  -2c_{1,\alpha}\int_{\xi>0} f(\xi) \xi^{1-2\alpha} \, \idiff \xi  \ .\end{equation} 
Moreover, let 
$ \lambda = \rbracket{-c_\ell (2-2\alpha)}^{\frac{1}{2-2\alpha}} $
then $(\Omega_\lambda(x),\widetilde{c_\ell },\widetilde{c_\omega}) = (\lambda^{-1}\Omega(\lambda x), -\frac{1}{2-2\alpha}$, $  -\frac{1}{2-2\alpha}) $ is a solution to \eqref{Equation: SelfSimilarProfileOmega} with $\widetilde{c_\omega} + (1 - 2\alpha) \widetilde{c_\ell } = -1$.

Conversely, if $\Omega=-xf$ is a solution to \eqref{Equation: SelfSimilarProfileOmega} such that $f$ is an even function of $x$, $f(x)\geq0$ and $f'(x)\leq0$ for $x\in[0,\infty)$, $f(0)=1$, $\lim_{x\to+\infty} f(x)=0$, and $\lim_{x\to0+}\frac{f'(x)}{2x}=-1$ (by re-normalization), then $f$ is a fixed-point of $\scrR_\alpha$, and $c_\ell $ is related to $f$ as in \eqref{Equation: c_l}.
\end{proposition}
\begin{proof} The first statement follows directly from the construction of $\scrR_\alpha$. The claims that $\lim_{x\to+\infty} f(x)=0$ and $b(f)<+\infty$ provided $f=\scrR_\alpha(f)$ will be proved in Lemma \ref{Lemma: CompactSupportfixPoint} below.  The formula for $c_\ell $ is obtained by comparing the definition of $\scrR_\alpha$ and the expression of $v'(0)$ in \eqref{Equation: fixedPointEqn}, and the formula of $c_\omega$ follows from \eqref{Equation: c_l}. 

Conversely, if $\Omega=-xf$ is a solution to equations \eqref{Equation: SelfSimilarProfileOmega} then we can rescale $\Omega$ by $\Omega_\lambda(x) = \lambda^{1-2\alpha}\Omega(\lambda x)$ so that $\lim_{x\to0+} \frac{f'(x)}{2x}=-1$. Notice that $$ \lim_{x\to 0} \frac{\scrT_\alpha(f)'(x)}{-2x} = -\frac{1+2\alpha}{3}\int_{\xi>0}\frac{f'(\xi)}{\xi^{2\alpha}} \, \idiff \xi = c_{1,\alpha} \frac{2\alpha(1+2\alpha)}{3} \int_{\xi>0} \frac{1-f(\xi)}{\xi^{1+2\alpha}} \, \idiff \xi \ . $$
Thus $\lim_{x\to0+} \frac{\scrR_\alpha f'(x)}{2x}=\lim_{x\to0+} \frac{f'(x)}{2x}=-1$ gives us 
\begin{align*}
c(f) & =  -\frac{1+2\alpha}{3} c_{1,\alpha} \int_{\xi>0} \xi^{-2\alpha}f'(\xi)\, \idiff \xi  = \frac{2\alpha(1+2\alpha)}{3} c_{1,\alpha} \int_{\xi>0} \frac{1-f(\xi)}{\xi^{1+2\alpha}}  \, \idiff \xi \ , \\
c_\ell  &= c(f) - b(f) = \frac{2\alpha(1+2\alpha)}{3} \int_{\xi>0} \frac{1-f(\xi)}{\xi^{1+2\alpha}}  \, \idiff \xi  -2c_{1,\alpha}\int_{\xi>0} f(\xi) \xi^{1-2\alpha} \, \idiff \xi \ .
\end{align*}
\end{proof}


\subsubsection{ Roadmap for proofs }
The remainder of this section is devoted to proving the existence of fixed points of $\scrR_\alpha$ on $\calV_1$. Before getting into the details, let us briefly explain the design of the set $\calV_1$ and the ideas behind the proof. In order to apply the Schauder fixed-point theorem, we want that 
\begin{enumerate}
\item $\calV_1$ is nonempty, convex, and closed in the underlying Banach topology;
\item $\calV_1$ is a compact set in this Banach space; 
\item $\scrR_\alpha$ maps $\calV_1$ continuously into itself in the same topology;
\end{enumerate} We note that (1) is automatically satisfied by the design of $\calV_1$. To establish (2) and (3), it is crucial to observe that the intermediate linear map $\scrT_\alpha$ preserves monotonicity and square-root-convexity on $[0,\infty)$, which can be proven by applying integration by parts to the formula of $\scrT_\alpha$. This monotonicity and convexity preserving property of $\scrT_\alpha$ then passes on to the non-linear map $\scrR_\alpha$  through some straightforward calculations of derivatives, which provide powerful controls on $\scrT_\alpha$. In particular, it implies the uniform estimates that $\max(0,1-x^2)\leq \scrR_\alpha(f)(x) \leq1$ and that $\scrR_\alpha(f)'(1/2)\leq-\eta$. Also, it is easy to see that $\scrR_\alpha(f)$ is even and $\scrR_\alpha(f)(0)=1$. Thus, the function set $\calV_1$ is closed under the map $\scrR_\alpha$. Moreover, the monotonicity and convexity properties lead to the continuity of $\scrR_\alpha:\calV_1\to\calV_1$ and the compactness of $\calV_1$ in the $L^\infty_\rho$-topology. Finally, with all these ingredients in hand, we can apply the Schauder fixed-point theorem on $\scrR_\alpha:\calV_1\to\calV_1$ to conclude the proof.


\subsubsection{Properties of $c(f)$} We start with a finer estimate of $c(f)$ that will be useful later.

\begin{lemma}\label{Lemma: c(f)} For any $\alpha\in(0,1)$,  any $f\in\calV_1$, and any $x>0$, we have $$c(f) \leq \frac{1+2\alpha}{3(1-\alpha)} c_{1,\alpha} (1-f(x))^{\min(1-\alpha,\alpha)} (1+x^{-2\alpha}) \ .$$
\end{lemma}
\begin{proof}
Let us fix $x>0$. Notice that if $\xi<x$, we have $f(\xi)\geq \max(1-\xi^2, f(x))$, and thus $1-f(\xi)\leq \min(\xi^2, 1-f(x))$. Meanwhile, if $\xi>x$, then convexity of $f(\sqrt{\cdot})$ implies that $\frac{1-f(\xi)}{\xi^2}\leq \frac{1-f(x)}{x^2}$, and thus $1-f(\xi)\leq \min( 1, \xi^2 x^{-2}(1-f(x)) )$. Combining both cases together gives us
 $$1-f(\xi) \leq \min(\xi^2, 1-f(x)) + \min( \xi^2 x^{-2} (1-f(x)) ,1) \ ,$$ so we can compute that
\begin{align*}
c(f) & \leq \frac{2\alpha(1+2\alpha)}{3} c_{1,\alpha} \int_{\xi>0} \frac{\min(\xi^2, 1-f(x))}{\xi^{1+2\alpha}} + \frac{\min(\xi^2 x^{-2}(1-f(x)),1)}{\xi^{1+2\alpha}}\, \idiff \xi \\
& = \frac{1+2\alpha}{3(1-\alpha)} c_{1,\alpha} \sbracket{(1-f(x))^{1-\alpha} + \rbracket{\frac{1-f(x)}{x^2}}^\alpha } \\
& \leq \frac{1+2\alpha}{3(1-\alpha)} c_{1,\alpha} (1-f(x))^{\min(1-\alpha,\alpha)}(1+x^{-2\alpha}) \ .
\end{align*}
\end{proof}

\begin{lemma}\label{Lemma: cHolderContinuity} $c(\cdot): \calV_1 \to\bbR $ is H\"older continuous with exponent $1-\alpha$ in the $L^\infty_\rho$-norm. In particular, we have
$$ \abs{ c(f_1) - c(f_2) } \leq C \norm{\rho(f_1 - f_2)}_{L^\infty}^{1-\alpha}   \ ,$$
for any $f_1, f_2 \in\calV_1$.
\end{lemma}
\begin{proof} We recall that $\rho(x) = (1+x)^{-\alpha}$. Denote $\delta = \norm{\rho (f_1-f_2)}_{L^\infty}\leq1$. Since $f_1(x),f_2(x)\geq \max(0,1-x^2)$, we have 
$$ \abs{f_1(x) - f_2(x)}\leq \min(x^2, \delta(1+|x|)^\alpha) \ .$$
Hence, we obtain
\begin{align*} 
& \abs{c(f_1) - c(f_2) } \leq \frac{2\alpha(1+2\alpha)}{3}c_{1,\alpha}\int_{\xi>0} \frac{ \abs{f_1(\xi) - f_2(\xi)} }{\xi^{1+2\alpha}}\idiff\xi \\
\leq & \frac{2\alpha(1+2\alpha)}{3} c_{1,\alpha} \rbracket{ \int_{\xi<\sqrt{\delta}} \xi^{1-2\alpha}\, \idiff \xi + \int_{\xi>\sqrt{\delta}} \frac{\delta(1+|\xi|)^\alpha}{\xi^{1+2\alpha}}\, \idiff \xi  } \\
\leq & \frac{(1+2\alpha)(3-2\alpha)}{3(1-\alpha)}c_{1,\alpha} \delta^{1-\alpha} \ .
\end{align*}
\end{proof}


\subsubsection{Properties of $\scrT_\alpha$ and $\scrR_\alpha$ } We now turn to study the intermediate maps $\scrT_\alpha$ and $\scrR_\alpha$. As an important observation in our fixed-point method, they preserve the monotonicity and convexity of functions in $\calV_1$.
\begin{proposition}\label{Proposition:RaV1R2} For any $\alpha\in(0,1)$, operator $\scrR_\alpha$ maps $\calV_1$ to itself.
\end{proposition}
\begin{proof}
We can directly compute that
\begin{align*}
&\scrT_\alpha(f)(0) = 0, \quad \quad \scrR_\alpha(f)(0) = 1 \ , \\
& \scrR_\alpha(f)(x) = \max\rbracket{0, 1 + \frac{\scrT_\alpha(f)(x)}{c(f)}  }\geq 0 \ , \\
&\scrR_\alpha(f)'(x)  = \frac{c_{1,\alpha}}{c(f)} \int_0^\infty f'(\xi) \xi^{1-2\alpha} F_{1,1-2\alpha}'(x/\xi)\, \idiff \xi \leq 0 \ .
\end{align*}
Notice that the convexity of $f(\sqrt{\cdot})$ is equivalent to $(f'(x)/x)'\geq0$ on $x\in(0,\infty)$, and then we are going to show that $({\scrR_\alpha(f)'(x)}/x)' = c(f)^{-1}({\scrT_\alpha(f)'(x)}/x)'\geq 0$. Let $\gamma=1-2\alpha$, we have
\begin{align*}
& \frac{\scrT_\alpha(f)'(x)}{x} = c_{1,\alpha} \int_{\xi>0} \frac{f'(\xi)}{\xi} \frac{\xi^{1+\gamma}}{x} F_{1,\gamma}'(x/\xi)\, \idiff \xi \\
= & c_{1,\alpha} \int_{\xi>0} \frac{f'(\xi)}{\xi} \partial_{\xi}\rbracket{\xi^{1+\gamma}F_{2,\gamma}(x/\xi)} \, \idiff \xi + \frac{2(2-\gamma)}{3} c_{1,\alpha}\int_{\xi>0} \frac{f'(\xi)}{\xi^{1-\gamma}} \idiff\xi\\
= & c_{1,\alpha} \int_{\xi>0} \frac{f'(\xi)}{\xi} \rbracket{ \frac{\xi^{1+\gamma}}{x} F_{1,\gamma}'(x/\xi) - \frac{2(2-\gamma)\xi^\gamma}{3} }\, \idiff \xi + \frac{2(2-\gamma)}{3} c_{1,\alpha}\int_{\xi>0} \frac{f'(\xi)}{\xi^{1-\gamma}} \idiff\xi\\
= & c_{1,\alpha} \int_{\xi>0} \rbracket{\frac{f'(\xi)}{\xi} }' \xi^{1+\gamma} F_{2,\gamma}(x/\xi)\, \idiff \xi + \frac{2(2-\gamma)}{3} c_{1,\alpha}\int_{\xi>0} \frac{f'(\xi)}{\xi^{1-\gamma}}  \idiff\xi\ .
\end{align*}
Therefore, we get
$$
 \rbracket{\frac{\scrT_\alpha(f)'(x)}{x}}' =  c_{1,\alpha} \int_{\xi>0} \rbracket{\frac{f'(\xi)}{\xi} }' \xi^{\gamma} F_{2,\gamma}' (x/\xi)\, \idiff \xi \geq 0 
\ . $$

Notice that $\lim_{x\to0} \frac{\scrR_\alpha(f)'(x)}{2x} =-1 $ and convexity of $\scrR_\alpha(f)(\sqrt{x})$ shows that $\scrR_\alpha(f)(x)\geq \max(0, 1-x^2).$ Moreover, when $x\geq\frac12$, we can compute that 
\begin{equation}\label{Equation: Ra1}\begin{split}
\scrR_\alpha(f)'(x) & = - \frac{c_{1,\alpha}}{c(f)} \int_0^\infty f'(\xi) \xi^{1-2\alpha} F_{1,1-2\alpha}'(x/\xi)\, \idiff \xi \\
& \leq -  \frac{c_{1,\alpha}}{c(f)} \int_0^x f'(\xi) \xi^{1-2\alpha} F_{1,1-2\alpha}'(x/\xi)\, \idiff \xi \\
& \leq -  \frac{c_{1,\alpha}}{c(f)} \int_0^x \frac{f'(x)}{x} \xi^{2-2\alpha} F_{1,1-2\alpha}'(x/\xi)\, \idiff \xi \\
& \leq -  \frac{c_{1,\alpha} }{c(f)} x^{2-2\alpha} f'(x) \rbracket{ \int_0^1 t^{2-2\alpha} F_{1,1-2\alpha}'(1/t) \idiff t } \\
& =  \frac{c_{1,\alpha} }{c(f)} x^{2-2\alpha} f'(x)  \frac{2^{1-2\alpha}(1+2\alpha)}{(1-\alpha)(3-2\alpha)(5-2\alpha)} \ .
\end{split}\end{equation}

We recall from the convexity of $f(\sqrt{x})$ implies that $f(\frac12)-1\geq \frac12f'(\frac12)\geq -\frac{\eta}{2}$. This together with Lemma \ref{Lemma: c(f)} implies that $$c(f) \leq \frac{1+2\alpha}{3(1-\alpha)} c_{1,\alpha} \rbracket{\frac{\eta}{2}}^{\min(1-\alpha,\alpha)} (1+2^{2\alpha}) \ , $$
and hence we conclude that
$$ \frac{\abs{\scrR_\alpha(f)'(\frac12)}}{\abs{f'(\frac12)}}\geq \frac{3}{2(3-2\alpha)(5-2\alpha)(1+4^\alpha) (\eta/2)^{\min(1-\alpha,\alpha)}} = 1 \ .$$
\end{proof}

Next, we show that $\scrR_\alpha$ is continuous on $\calV_1$ in the $L^\infty_\rho$-topology.
\begin{proposition}\label{Proposition: ContinuityMapRa} For any $\alpha\in(0,1),\scrR_\alpha: \calV_1\to\calV_1$ is continuous with respect to the $L^\infty_\rho$-norm.
\end{proposition}
\begin{proof}
We recall that $\rho(x)=(1+|x|)^{-\alpha}$. Given any (fixed) $f_0\in\calV_1$, we only need to prove that $\scrR_\alpha$ is $L^\infty_\rho$-continuous at $f_0$. Denote $g_0:=\scrR_\alpha(f_0)$. Let $\varepsilon>0$ be an arbitrarily small number. Since $\scrR_\alpha(f_0)$ is bounded, continuous, and non-increasing on $[0,\infty)$, there is some $M>1$ such that
$$ \rho(M) g_0(M) = \varepsilon \ , $$
and this also means that $\rho(x)g_0(x) \leq \varepsilon$ for all $x\geq M$.
Let $f\in\calV_1$ be arbitrary, and similarly denote $g = \scrR_\alpha(f)$. Suppose that $\norm{\rho(f-f_0)}_{L^\infty}\leq \delta$ for some sufficiently small $\delta>0$. For any $x>0$, we have
\begin{align*}
 & \abs{ \scrT_\alpha(f)(x) - \scrT_\alpha(f_0)(x)  }  = c_{1,\alpha} \abs{\int_{\xi>0} (f(\xi) - f_0(\xi)) \partial_\xi(\xi^{2-2\alpha} F_{1,1-2\alpha}(x/\xi))\, \idiff \xi  } \\
 = & c_{1,\alpha} \abs{\int_{\xi>0} (f(\xi) - f_0(\xi)) ((2-2\alpha)\xi^{1-2\alpha} F_{1,1-2\alpha}(x/\xi) - \xi^{-2\alpha} x F_{1,1-2\alpha}'(x/\xi) ) \, \idiff \xi  } \\
 \leq & \delta c_{1,\alpha} \int_{\xi>0} (1+\xi)^\alpha ((2-2\alpha)\xi^{1-2\alpha} F_{1,1-2\alpha}(x/\xi) - \xi^{-2\alpha} x F_{1,1-2\alpha}'(x/\xi) ) \, \idiff \xi  \\
 = &  \delta c_{1,\alpha}  x^{2-2\alpha} \int_{\xi>0} (1+tx)^\alpha ((2-2\alpha)t^{1-2\alpha} F_{1,1-2\alpha}(1/t) - t^{-2\alpha} F_{1,1-2\alpha}'(1/t) ) \idiff t \\
 \leq &  \delta c_{1,\alpha} x^{2-2\alpha} (1+x)^\alpha \int_{t>0} (1+t)^\alpha ((2-2\alpha)t^{1-2\alpha} F_{1,1-2\alpha}(1/t) - t^{-2\alpha} F_{1,1-2\alpha}'(1/t) ) \idiff t  \\
 \lesssim & \delta x^{2-2\alpha} (1+x)^\alpha \ .
 \end{align*}
The last integral of $t$ above is finite since the integrand behaves like $t^{-1-\alpha}$ as $t\to\infty$. A similar argument shows that $|\scrT_\alpha(f)(x)|, |\scrT_\alpha(f_0)(x)|\lesssim x$. Combining these estimates with Lemma \ref{Lemma: c(f)} and Lemma \ref{Lemma: cHolderContinuity} yields
$$ |g(x) - g_0(x) | \lesssim \delta^{2-2\alpha} x (1+x)^\alpha + x \delta^{1-\alpha} \lesssim \delta^{1-\alpha} x (1+x)^\alpha \ , $$
which implies that for any $x\in[0,M] $,
$$ \rho(x) | g(x) - g_0(x) | \leq  \delta^{1-\alpha} M \ .$$
Again, provided that $\delta$ is sufficiently small, we have $|\rho(M)g(M)|\leq 2\varepsilon$. By the monotonicity of $\scrR_\alpha(f)$, we also have $|\rho(x) g(x)| \leq 2\varepsilon$ when $x\geq M$. As a consequence, we can choose $\delta$ sufficiently small so that 
$$ \norm{\rho(x)(\scrR_\alpha(f)(x) - \scrR_\alpha(f_0)(x) )}_{L^\infty} \lesssim \varepsilon \ , $$
for any $f\in\calV_1$ such that $\norm{\rho(x)(f(x)-f_0(x))}_{L^\infty}\leq\delta$. We have thus proved that $\scrR_\alpha$ is $L_\rho^\infty$-continuous at $f_0$ as $\varepsilon$ is arbitrary.
\end{proof}


\subsubsection{Existence of fixed-point $f_*$ of $\scrR_\alpha$} One last ingredient for establishing existence of fixed points of $\scrR_\alpha$ is the compactness of the set $\calV_1$. 
\begin{lemma}\label{Lemma: CompactSetV1} The set $\calV_1$ is compact with respect to the $L^\infty_\rho$-norm.
\end{lemma}
\begin{proof}
For any $f\in\calV_1$, we use convexity and monotonicity to obtain
$$ -\frac{f'(x)}{2x} \leq \frac{f(0)-f(x)}{x^2}\leq \min(1,x^{-2}) \ , $$
which implies that $|f'(x)|\leq\min(2x,2x^{-1})\leq2$. Based on this, we are going to show that $\calV_1$ is sequentially compact.
Let $\{f_n\}_{n=1}^\infty$ be an arbitrary sequence in $\calV_1$. Initialize $n_{0,k}=k,k\geq1$.  For each integer $m\geq1$, let $\varepsilon_m=2^{-m}$ and $L_m= \varepsilon_m^{-2}$. It follows that $\rho(x) f_n(x)\leq \rho(x)\leq \varepsilon_m$ for all $x\geq L_m$. Furthermore, since $|f'_n(x)|\leq2$ on $[0,L_m]$, we can apply the Arzela-Ascoli theorem to select a sub-sequence $\{f_{n_m,k}\}_{k=1}^\infty$ of $\{f_{n_{m-1},k}\}_{k=1}^\infty$ such that $\norm{\rho (f_{n_m,i} - f_{n_m,j})}\leq 2\varepsilon_m$ for any $i,j\geq1$. Then the diagonal sub-sequence $\{f_{n_m,m}\}_{m=1}^\infty$ is a Cauchy sequence in the $L^\infty_\rho$-norm. This proves that $\calV_1$ is sequentially compact.
\end{proof}

We are now ready to prove the existence of fixed points of $\scrR_\alpha$ for any $\alpha\in(0,1)$ using the Schauder fixed-point theorem.
\begin{theorem} For each $\alpha\in(0,1)$, the map  $\scrR_\alpha : \calV_1\to\calV_1$ has a fixed point $f_*$. That is, $\scrR_\alpha(f_*)=f_*$. As a corollary, for each $\alpha\in(0,1)$, \eqref{Equation: SelfSimilarProfileOmega} admits a solution $(\Omega,c_\ell ,c_\omega)$ with $f=-\Omega/x\in\calV_1$ and $c_\ell ,c_\omega$ given in Proposition \ref{Proposition: relation1Dmodel2DSQGR2}.
\end{theorem}
\begin{proof}
By Proposition \ref{Proposition: ContinuityMapRa} and Lemma \ref{Lemma: CompactSetV1}, $\calV_1$ is convex, closed and compact in the $L^\infty_\rho$-norm, and $\scrR_\alpha$ continuously maps $\calV_1$ into itself. The Schauder fixed-point theorem implies that $\scrR_\alpha$ has a fixed point in $\calV_1$. The second part of the theorem then follows from Proposition \ref{Proposition: relation1Dmodel2DSQGR2}.
\end{proof}


\subsubsection{Properties of fixed-point $f_*$ of $\scrR_\alpha$}

We note that the map $\scrR_\alpha$ involves a positive-part truncation and the construction is designed to allow a free boundary. The following lemma shows that a nontrivial fixed point cannot stay strictly positive forever, and this is consistent with the heuristic idea that an expanding-type self-similar blow-up profile often corresponds to a compact support.

\begin{lemma}\label{Lemma: CompactSupportfixPoint} Let $f_*\in\calV_1$ denote a fixed-point of $\scrR_\alpha$. Then $f_*$ is compactly supported. As a corollary, we know that $c_\ell <0$ and $c(f_*)<b(f_*)<\infty$, and hence $f_*$ corresponds to an expanding-type blow-up.
\end{lemma}
\begin{proof}
Suppose by contrast that $f_*$ is not compactly supported, then for any $x>0$, one has $f_*'(x) = \scrR_\alpha(f_*)'(x) = \frac{c_{1,\alpha}}{c(f_*)} \int_0^\infty f_*'(\xi) \xi^{1-2\alpha} F_{1,1-2\alpha}'(x/\xi)\, \idiff \xi < 0 $ and thus we use the result from \eqref{Equation: Ra1} to obtain
$$ 1 = \frac{\scrR_\alpha(f_*)'(x)}{f_*'(x)} \geq \frac{c_{1,\alpha}}{c(f_*)} \frac{2^{1-2\alpha}(1+2\alpha)}{(1-\alpha)(3-2\alpha)(5-2\alpha)} x^{2-2\alpha}  \ ,$$
should hold for any $x\geq 1/2$, which is impossible.
\end{proof}

To obtain higher regularity of $f$ or $\Omega$, we will show in the following lemma on the improved regularity for operator $\scrT_\alpha$. The argument says that $x\scrT_\alpha(f)$ has extra $2-2\alpha$ order of local regularity than $xf$.
\begin{lemma}\label{Lemma: RegMapTa1} Let $f\in\calV_1$ and assume that $b(f)<\infty$. Assume that $f$ is compactly supported and let $[-L,L]$ denote $\supp f$. Let $s\geq0$ be some fixed number. Let $\chi_0,\chi_1\in C_c^\infty((-L,L))$ be such that $\chi_1\equiv 1$ on a neighborhood of $\supp\chi_0$.
If
\[
(-\Delta)^{\frac{s}{2}}(xf(x)\chi_1)\in L^2_{loc}(\bbR),
\]
then
\[
(-\Delta)^{\frac{s}{2}+1-\alpha}(x \scrT_\alpha(f)\chi_0) \in L^2_{loc}(\bbR).
\]\end{lemma}
\begin{proof} 
We note that $x\scrT_\alpha(f)\chi_0 =\calP_{0,\alpha}(-xf)\chi_0 + (c_\ell  - b(f))x\chi_0$. Hence, we have 
$$  (-\Delta)^{\frac{s}{2}+1-\alpha}(x\scrT_\alpha(f)\chi_0) =  (-\Delta)^{\frac{s}{2}+1-\alpha}(\calP_{0,\alpha}(-xf)\chi_0) +  (-\Delta)^{\frac{s}{2}+1-\alpha}((c_\ell  - b(f))x\chi_0 ) \ .$$
We note that $(c_\ell  - b(f))x\chi_0$ in the last term is in $C_c^\infty(\bbR)$, and therefore
$$ (-\Delta)^{\frac{s}{2}+1-\alpha}((c_\ell  - b(f))x\chi_0) \in C^\infty(\bbR) \subseteq L^2_{loc} \ . $$
Therefore it remains to prove the claim for $\chi_0 \calP_{0,\alpha}(-xf)$. Let $\chi_1\in C_c^\infty((-L,L))$ such that $\chi_1(x)\equiv1$ on a neighborhood of $\supp\chi_0$.
 We split $\calP_{0,\alpha}(-xf)\chi_0$ into two terms, where one corresponds to the near component and the other corresponds to the far component:
 $$ \chi_0 \calP_{0,\alpha}(-xf) = \chi_0 \calP_{0,\alpha}(-xf\chi_1) + \chi_0 \calP_{0,\alpha}(-xf(1-\chi_1)) \ .$$
We claim that the far component $\chi_0 \calP_{0,\alpha}(-xf(1-\chi_1))$ is smooth. Since
\[
\text{dist}(\supp \chi_0,\supp(1-\chi_1))>0,
\]
the kernel of $\calP_{0,\alpha}$ is $C^\infty$ in $x$ on $\supp\chi_0$, uniformly for
$\xi\in \supp(1-\chi_1)$. Therefore
\[
\chi_0 \calP_{0,\alpha}(-xf(1-\chi_1))\in C_c^\infty(\bbR).
\]
Since $(-\Delta)^{\frac{s}{2}}$ maps $C_c^\infty(\bbR)$ continuously into $C^\infty(\bbR)$ for every $s\ge 0$,
we obtain
\[
(-\Delta)^{\frac{s}{2}+1-\alpha}\!\bigl(\chi_0 \calP_{0,\alpha}(-xf(1-\chi_1))\bigr)\in C^\infty(\bbR)\subset L^2_{loc}(\bbR).
\]
For the near component $\chi_0 \calP_{0,\alpha}(-xf\chi_1)$, we write
\[
(-\Delta)^{1-\alpha}\bigl(\chi_0 \calP_{0,\alpha}(-xf\chi_1)\bigr)
=
\chi_0 (-\Delta)^{1-\alpha}\calP_{0,\alpha}(-xf\chi_1)
+
[(-\Delta)^{1-\alpha},\chi_0]\calP_{0,\alpha}(-xf\chi_1).
\]
We note that $(-\Delta)^{1-\alpha}\calP_{0,\alpha}$ is the identity map up to the fixed scalar $C_\alpha\neq 0$, so
the first term equals $-C_\alpha\,xf\chi_0$, because $\chi_1\equiv 1$ on $\supp\chi_0$.
Applying $(-\Delta)^{s/2}$, we obtain
\begin{align*}
& (-\Delta)^{\frac{s}{2}+1-\alpha}\bigl(\chi_0 \calP_{0,\alpha}(-xf\chi_1)\bigr) \\
= & 
-C_\alpha\,(-\Delta)^{s/2}(xf\chi_0)
+
(-\Delta)^{s/2}[(-\Delta)^{1-\alpha},\chi_0]\calP_{0,\alpha}(-xf\chi_1) \\
= &
-C_\alpha\,\chi_0(-\Delta)^{s/2}(xf\chi_1)
-C_\alpha[(-\Delta)^{s/2},\chi_0](xf\chi_1) \\
&
+(-\Delta)^{s/2}[(-\Delta)^{1-\alpha},\chi_0]\calP_{0,\alpha}(-xf\chi_1).
\end{align*}

We now treat the three terms on the right-hand side.

First, since $\chi_1\equiv 1$ on a neighborhood of $\supp\chi_0$, the hypothesis
\[
(-\Delta)^{s/2}(xf\chi_1)\in L^2_{loc}(\bbR)
\]
implies
\[
\chi_0(-\Delta)^{s/2}(xf\chi_1)\in L^2_{loc}(\bbR).
\]

Second, by Lemma~\ref{Lemma:FracCommCutoff} (applied to $[(-\Delta)^{s/2},\chi_0]$ with $\gamma=s$ and $t=1$),
the commutator $[(-\Delta)^{s/2},\chi_0]$ maps $H^s_{loc}$ into $H^1_{loc}\subset L^2_{loc}$.
Since $xf\chi_1\in H^s_{loc}$ by the hypothesis, it follows that
\[
[(-\Delta)^{s/2},\chi_0](xf\chi_1)\in L^2_{loc}(\bbR).
\]

Third, we use the same method to show that $(-\Delta)^{s/2}[(-\Delta)^{1-\alpha},\chi_0]\calP_{0,\alpha}$ is a pseudo-differential operator of order $s-1$, and then we obtain
\[
(-\Delta)^{s/2}[(-\Delta)^{1-\alpha},\chi_0]\calP_{0,\alpha}(-xf\chi_1)\in L^2_{loc}(\bbR).
\]

Hence
\[
(-\Delta)^{\frac{s}{2}+1-\alpha}\bigl(\chi_0 \calP_{0,\alpha}(-xf\chi_1)\bigr)\in L^2_{loc}(\bbR) \ . \qedhere
\]
\end{proof}

Thus we have the following result on the smoothness for the fixed-point $f_*$ of operator $\scrR_\alpha$.
\begin{proposition}\label{Proposition: SmoothnessProfile}
For any $\alpha\in(0,1)$, let $f_*\in\calV_1$ be a fixed-point of $\scrR_\alpha$. Let $L := \sup\{x: f_*(x)>0\}$. Then $f_*$ is compactly supported on $[-L,L]$ and $f_*$ is smooth in the interior  $(-L,L)$.
\end{proposition}
\begin{proof}
From Lemma \ref{Lemma: CompactSupportfixPoint}, we know that $f_*$ has compact support. Since $f_*\in\calV_1$, $f_*$ is continuous.  Since $\scrR_\alpha(f) = \max\rbracket{0, 1 + c(f)^{-1} \scrT_\alpha(f) }$, we use Lemma \ref{Lemma: RegMapTa1} repeatedly to improve regularity of $\omega_*(x)=-xf_*(x)$ and show that $\omega_*$ is smooth in the interior  $(-L,L)$. As a consequence, $f_*$ is smooth in the interior $(-L,L)$.
\end{proof}

\section{Generalized SQG  on $\HP$}\label{Section: SQGHP}
A recurring intuition in 2D fluid problems is that boundaries can create a hyperbolic flow geometry: fluid can be driven toward a distinguished point along the boundary while being expelled in the normal direction, creating a competition between compression and ejection. In the gSQG setting on the upper half-plane $\bbR^2_+$ (with the natural odd reflection formulation for Dirichlet boundary data), this geometry motivates searching for scenarios where inward compression dominates and produces singular growth. The half-plane problem is therefore a natural testing ground for possible blow-up mechanisms. This perspective is not merely heuristic: The role of the solid boundary is to create scenarios in which the flow toward the origin is stronger than the flow away from it, which is the mechanism underlying double-exponential growth in certain 2D Euler boundary scenarios. This suggests that boundaries may also facilitate singularity formation in SQG-type models.  There are now rigorous results indicating that boundaries can indeed change the picture in a decisive way. For example, local well-posedness in boundary-adapted classes for $\alpha\in(0,\frac14]$ on the half-plane is proven in \cite{Zlato2023LocalRA}, and the same work also constructs solutions exhibiting finite-time blow-up throughout that regime, with sharp ill-posedness beyond $\alpha>\frac14$.

To probe more robust blow-up scenarios, we next consider gSQG on the upper half-plane $\HP$, where the presence of a solid boundary can create a hyperbolic flow geometry: fluid is driven toward the origin along the boundary while being expelled along the vertical axis, opening the possibility that the inward compression dominates and produces singular growth, as happens in related boundary-driven mechanisms for 2D Euler.  We formulate the half-plane problem via odd reflection and develop the corresponding dynamic rescaling, so that steady states of the rescaled system represent candidate self-similar blow-up profiles for the original equation; in parallel, we derive a boundary-based one-dimensional reduction that again captures the leading-order singular behavior and serves as a tractable setting for a fixed-point construction of blow-up profiles consistent with numerical simulations.

We look at the (inviscid) generalized surface quasi-geostrophic (gSQG) on $\HP$:
\begin{equation}\label{Equation: gSQG-HP}
\begin{cases}
    \pt \theta(x,y,t)  + u(x,y,t)\cdot\nabla \theta(x,y,t) = 0 \text{ on } \HP \\
    u(x,y,t) = \nabla^\perp \NLIalp \overline{\theta}(x,y,t) 
    \end{cases} \ , \tag{gSQG-$\HP$}
\end{equation}
where $\alpha\in(0,\frac12)$ and $\overline{\theta}(x,y,t)$ is the extension of $\theta(x,y,t)$ from $\HP$ to $\bbR^2$ via odd reflection:
$$ \overline{\theta}(x,y,t)=
\begin{cases}
\theta(x,y,t) & \text{ if } y\geq 0 \\
-\theta(x,-y,t) & \text{ if } y < 0
\end{cases} \ .
$$
  
  The velocity $u = (u_1,u_2)$ can be explicitly expressed as the following singular integral:
  \begin{align*}  
        u_1(x,y,t) & = c_{0,\alpha} \iint_{(\xi,\eta)\in\HP}  \frac{(y-\eta)  \theta(\xi,\eta,t) }{((x-\xi)^2+(y-\eta)^2)^{1+\alpha}}  -  \frac{(y+\eta)  \theta(\xi,\eta,t) }{((x-\xi)^2+(y+\eta)^2)^{1+\alpha}}     \idiff\xi\diff\eta \ ,  \\
    u_2(x,y,t) & =- c_{0,\alpha}  \iint_{(\xi,\eta)\in\HP}  \frac{(x-\xi)  \theta(\xi,\eta,t) }{((x-\xi)^2+(y-\eta)^2)^{1+\alpha}}  - \frac{(x-\xi)  \theta(\xi,\eta,t) }{((x-\xi)^2+(y+\eta)^2)^{1+\alpha}}  \idiff\xi\diff\eta \ .
\end{align*}


\subsection{1D reduction for \eqref{Equation: gSQG-HP}}
In particular, we have $u_2(x,y,t)=0$ when $y=0$. Formally, let us consider a simplified model by taking $\theta$ independent of $y$. Thus we consider the following ansatz:
$$ \theta(x,t) = (T-t)^{c_\theta} \Theta\rbracket{ \frac{x}{(T-t)^{c_\ell }}}  \ . $$
Substituting above ansatz into \eqref{Equation: gSQG-HP} and consider the case $y=0$ yields
$$ -c_\theta (T-t)^{c_\theta-1} \Theta(z) + (T-t)^{c_\theta-1} c_\ell  z  \Theta'(z) + (T-t)^{2c_\theta -2\alpha c_\ell  } U(z) \Theta'(z) = 0 \ ,$$
where $z = x/ (T-t)^{c_\ell }$ and 
\begin{align*} U(x) & = -2c_{0,\alpha} \iint_{\xi\in\bbR,\eta>0} \frac{\eta\Theta(\xi)}{((x-\xi)^2+\eta^2)^{1+\alpha}}  \idiff\xi\diff\eta   = \frac{2c_{0,\alpha}}{-2\alpha} \int_{\xi\in\bbR}   \frac{\Theta(\xi) }{|x-\xi|^{2\alpha}}     \idiff \xi  \  .
\end{align*} Balancing the above equation yields $ c_\theta - 2\alpha c_\ell  + 1 =0  $
and an equation for the self-similar profile:
\begin{equation}\label{Equation: SelfSimilarProfileOmegaHP} (c_\ell  x + U(x)) \partial_x\Theta(x)  = c_\theta \Theta(x)  \ .  \end{equation}
We note that if $(\Theta(x),c_\ell ,c_\theta)$ is a solution to above equation, then so is
\begin{equation}\label{Equation: scaling invariant HP}
(\Theta_{\lambda,\gamma}(x), c_{l,\lambda,\gamma}, c_{\omega,\lambda,\gamma}) = ( \gamma \Theta(\lambda x), \gamma \lambda^{2\alpha}  c_\ell , \gamma \lambda^{2\alpha} c_\theta )  \ .
\end{equation}
This means that we can relax the restriction $c_\theta - 2\alpha c_\ell  = -1  $ to any $c_\theta - 2\alpha c_\ell  < 0 $. In fact, for any $(c_\theta, c_\ell )$ such that $c_\theta - 2\alpha c_\ell  < 0 $, one can define
$$ (\widetilde{c_\theta},\widetilde{c_\ell }) = \rbracket{\frac{c_\theta}{ 2\alpha c_\ell  - c_\theta} , \frac{c_\ell }{2\alpha c_\ell  - c_\theta} } \ , $$
which satisfies $\widetilde{c_\theta} -2\alpha\widetilde{c_\ell }=-1$.
 For simplicity, let us take $c_\theta=1$ and the only requirement for $c_\ell $ is $c_\ell  > (2\alpha)^{-1}$.

Furthermore, we look for solutions that satisfy the following conditions:
 \begin{itemize}
 \item Odd symmetry: $\Theta(x)$ is an odd function of $x$, i.e., $\Theta(-x) = - \Theta(x)$.
 \item Regularity: $\Theta\in H^1_{loc}(\bbR)$.
 \item Non-degeneracy: $\Theta'(0)\neq0$.
 \end{itemize}

In view of \eqref{Equation: scaling invariant HP}, we can take $\Theta'(0)=1$ without loss of generality. Thus, let us consider the following change of variable
$$ f(x) :=\frac{\Theta(x)}{x\Theta'(0)} = \frac{\Theta(x)}{x}, V(x) = 1- U'(0)+ \frac{U(x)}{x} \ ,$$
and one has 
$$  \frac{f'}{f} = \frac{1-V(x)}{xV(x)} \ .$$
As a consequence, we reduce the equation \eqref{Equation: SelfSimilarProfileOmegaHP} for the self-similar profile to the following form:

\begin{equation}\label{Equation: fixedPointEqnHP}
\begin{split}
   f(x) & = \exp\rbracket{ \int_0^x \frac{1-V(y)}{y V(y)} \idiff y} \ , \\
    U(x) & = \frac{2c_{0,\alpha}}{-2\alpha} \int_{\xi\in\bbR}   \frac{ \xi f(\xi) }{|x-\xi|^{2\alpha}}     \idiff \xi \ , \\
    V(x) & = 1- U'(0)+ \frac{U(x)}{x}  \ .
       \end{split} 
\end{equation}


\subsection{Existence of Solutions by a fixed-point method}
Our goal of this section is to show that the nonlinear system \eqref{Equation: fixedPointEqnHP} admits non-trivial solutions. We do so by converting the problem into a fixed-point problem of some nonlinear map. In detail, we are going to define a linear operator $\frT_\alpha(f)(x) := \frac{U(x)}{x}- U'(0)$ which simplifies \eqref{Equation: fixedPointEqnHP} to $f(x) = \exp(\int_0^x -\frac{\frT_\alpha(f)(y)}{y(1+\frT_\alpha(f)(y))}dy)$, and then the problem becomes finding a fixed-point of nonlinear map $\frR_\alpha(f):=\exp(\int_0^x -\frac{\frT_\alpha(f)(y)}{y(1+\frT_\alpha(f)(y))}dy)$. The full details of one-to-one correspondence between solutions to \eqref{Equation: fixedPointEqnHP} and fixed-points of $\frR_\alpha$ will be shown in Proposition \ref{Proposition: relation1Dmodel2DSQGHP}.

To show existence of a fixed-point of the proposed map $\frR_\alpha(f)$, we are going to use again the Schauder fixed-point theorem. 

\begin{remark} As we mentioned above, this fixed-point strategy follows the approach developed in \cite{HQWW24} to study the finite-time self-similar blow-up for gCLM. One introduces a nonlinear operator on a suitably chosen invariant, convex, and compact set of monotone “shape-controlled” functions, verifies that the operator maps this set into itself and is continuous, applies the Schauder fixed-point theorem to obtain a profile, and then bootstraps regularity. The monotonicity and convexity-type constraints are carefully constructed so that the relevant operators preserve qualitative shape properties and yield the compactness estimates required by Schauder’s theorem. The precise choice of constraints is, however, kernel-dependent: compared with gCLM, the gSQG reduction involves a different singularity order and far-field behavior, which necessitates working in a different topology (e.g., weighted control and different decay/barrier conditions) to establish invariance, continuity, and compactness for our nonlinear map $\frR_\alpha$. Accordingly, we select an appropriate Banach space $\scrV_1$ tailored to these operator features.  Furthermore, for the one-dimensional reduction of gSQG on $\HP$, the difference is that the profile is focusing-type and has no compact support, so the invariant set must build in decay envelopes and a different normalization mechanism. Also, the proof has extra technical constants because we need to control tails carefully.
\end{remark}

 To this end, we need to select an appropriate Banach function space $\scrV_1$ in which we can establish invariance, continuity, and compactness of our nonlinear map $\frR_\alpha$.


\subsubsection{Details of function set $\scrV_1$ as the invariant set for the fixed-point method.}
Consider the Banach space of even continuous functions,
$$  \scrV_0 = \{ \Omega \in \calC(\bbR) \cap L^\infty(\bbR) \ : \ \Omega(x) = \Omega(-x)  \} \ ,$$
endowed with $L^\infty$-norm. 

Let us define
 $$ c_{0,\alpha}''':=  \frac{   (4^{1 + \alpha} - 4 + \alpha - 4 \alpha^3) \Gamma(\alpha) }{  2 \pi\alpha(5 - 2 \alpha) (3 - 2 \alpha) (1 - 2 \alpha)  \Gamma(2-\alpha)}  \ , $$
 and let $t_0>1$ be sufficiently large such that 
 \begin{equation}\label{Equation: t0}
  t_0^{\frac12-\alpha} > \max\rbracket{\frac{1}{c_{0,\alpha}''' },\frac{1-2\alpha}{2\alpha c_{0,\alpha}''' }} , \ t_0^{\frac12+\alpha}> 2^{2^\alpha+2}\frac{\Gamma(2+\alpha)\Gamma(\frac12-\alpha)}{3\pi \alpha\Gamma(1-\alpha)} \ .
  \end{equation}
 Moreover, we take $\delta_l\geq\frac12$ such that
\begin{equation}\label{Equation: deltal}
\frac{2^{2\alpha} \Gamma(2+\alpha)\Gamma(\frac12-\alpha) }{3\pi\Gamma(1-\alpha)}\leq \frac{\Gamma(\delta_l+1) t_0^{2\alpha-1}}{\Gamma(\delta_l+\alpha+\frac12)}\leq 2 \frac{2^{2\alpha} \Gamma(2+\alpha)\Gamma(\frac12-\alpha) }{3\pi\Gamma(1-\alpha)} \ .
\end{equation}
We comment that this $\delta_l$ can be obtained: When $\delta_l$ is sufficiently large, the middle term $\frac{\Gamma(\delta_l+1) t_0^{2\alpha-1}}{\Gamma(\delta_l+\alpha+\frac12)}$ converges to $\rbracket{\frac{\delta_l+\alpha+\frac12}{t_0^2}}^{\frac12-\alpha}$, so it suffices to let both $\delta_l$ and $t_0$ be sufficiently large such that $\frac{\delta_l+\alpha+\frac12}{t_0^2} = \rbracket{\frac32 \frac{2^{2\alpha} \Gamma(2+\alpha)\Gamma(\frac12-\alpha) }{3\pi\Gamma(1-\alpha)}}^{\frac{1}{\frac12-\alpha}} $.

We define \begin{equation}\begin{split}
& f_l(x) := \rbracket{1+  t_0^{-2} x^2}^{-\delta_l},    \\
& f_u(x) := \min\rbracket{1, t_0^{\delta_u} x^{-\delta_u} }, \  \delta_u:= \frac{ c_{0,\alpha}'''   t_0^{\frac12-\alpha}}{ 1 +c_{0,\alpha}''' t_0^{\frac12-\alpha}} \in(1-2\alpha,1) \ .
\end{split}\end{equation}

Moreover, we consider a closed (in the $L^\infty$-norm) and convex subset of $\scrV_0$,
$$ \scrV_1 :=  \cbracket{ f\in\scrV_0 \, : \, 
\begin{split} 
& f(0)=1,f \text{ is nonnegative and non-increasing on } [0,\infty), \\
& f(\sqrt{x}) \text{ is convex in } x,0\leq f_{l}(x) \leq f(x) \leq f_{u} \leq 1, \ f'(t_0)\leq -t_0^{\alpha-\frac32}
\end{split} }   \ . $$

Here and below, we use $f_-'$ and $f_+'$ to denote the left and the right derivatives of a function $f$. The function set $\scrV_1$ will act as the invariant set for our fixed-point method. The purpose of upperbound $f_u$ and lowerbound $f_l$ is to enforce suitable decay so that we establish desired results on the singular integral. The specific values of coefficients in $f_u, f_l$ and the choice of $t_0$ are actually determined through a bootstrap argument that will be clear later.

We remark that, though a function $f\in\scrV_1$ is not required to be differentiable, the one-sided derivatives $f_-'(x)$ and $f_+'(x)$ are both well defined at every point $x$ by the convexity of $f(\sqrt{x})$ in $x$. In what follows, we will abuse notation and simply use $f'(x)$ for $f'_-(x)$ and $f_+'(x)$ in both the weak and strong sense. For example, when we write $f'(x)\leq C$, we mean both  $f'_-(x) \leq C$ and $f'_+(x) \leq C$ at the same time. In this context, the non-increasing property of $f$ on $[0,\infty)$ can be represented as $f'\leq0$ when $x\geq0$.

The reason to choose this function set $\scrV_1$ is as follows: the assumption $f(0)=1$ is chosen by scaling invariance in \eqref{Equation: scaling invariant HP}. $f$ being nonnegative is a natural assumption in finding a simple formulation of a self-similar profile. It is also possible to search for a self-similar profile which changes signs but that will be more difficult to estimate. The monotonicity of $f$ ensures the kernel acts with a definite sign, prevents oscillations, makes many estimates easy. The square-root convexity, or convexity after reparameterization, is the key structural condition that makes the integral operator preserve shape, and it provides control of one-sided derivatives even without smoothness, which allows us to bootstrap regularity later. The properties of monotonicity of $f$ and convexity of $f(\sqrt{x})$ are also preserved by the map $\frR_\alpha$, see proofs in Proposition \ref{Proposition: RalphaConvexity}. The upperbound $f_u$, together with the lowerbound $f_l$ and the derivative bound $f'(t_0)\leq -t_0^{\alpha-\frac32}$, give us very delicate control in Lemma \ref{Lemma: bcbound}  on the values of key functionals $\frak{b}(f),\frak{c}(f):\scrV_1\to\bbR$ that will be defined afterwards. 

We note that the lowerbound $f_l$ here differs from $\max(0,1-x^2)$ in the previous section, and we impose an explicit upperbound $f_u$ with power-law decay. This is related to the half-plane mechanism which produces a profile with a long tail, so the function class must encode decay to make the operator well-defined and compactness possible. Another related heuristic idea is that the focusing-type blow-up that we are searching for in this section often corresponds to the case that the self-similar profile is not compactly supported. Overall, this power-law decay encoded in the upperbound $f_u$ makes it unnecessary to use the weight 
$\rho$ from the previous section, because the explicit tail control in the upperbound $f_l$ is sufficient to guarantee compactness of $\scrV_1$ on an unbounded domain. 


\subsubsection{ Introducing maps $\frT_\alpha(f)$ and $\frR_\alpha(f)$. }
To simplify the expression $V(x) = \frac{U}{x} + 1 - U'(0)$, we will separate some terms in the singular integral. We define a new operator $\frT_\alpha$ as follows 
\begin{equation}
\begin{split}
\frT_\alpha(f)(x) & =\frac{U}{x}-U'(0) =  \frac{2c_{0,\alpha}}{-2\alpha} \int_{\xi\in\bbR} \frac{\xi f(\xi) }{x|x-\xi|^{2\alpha}} - 2\alpha\frac{f(\xi)}{|\xi|^{2\alpha}} \, \idiff \xi \\
& = \frac{2c_{0,\alpha}}{-2\alpha} \int_{\xi>0} \frac{\xi f(\xi)}{x} \rbracket{|x-\xi|^{-2\alpha} - |x+\xi|^{-2\alpha} } - 4\alpha \frac{f(\xi)}{|\xi|^{2\alpha}} \, \idiff \xi \\
& = - 2c_{0,\alpha} \int_{\xi>0} f'(\xi) \xi^{1-2\alpha} F_{1,-2\alpha}(x/\xi) \, \idiff \xi  \ .
\end{split}
\end{equation}
where the definition of auxiliary functions $F_{1,\gamma}$ can be found in Appendix \ref{Section: Auxiliary functions}. This also gives the following equivalent expression for \eqref{Equation: fixedPointEqnHP}:
$$ f(x) =  \exp\rbracket{-\int_0^x \frac{\frT_\alpha(f)(y)}{y (1+\frT_\alpha(f)(y))} \, dy }  \ .$$

This definition of $\frT_\alpha$ only uses the integral on $[0,\infty)$ since $f\in\scrV_1\subseteq\scrV_0$  is an even function. We will always employ this symmetry property in the sequel. One should note that $\frT_\alpha$ is well-defined for any $f\in\scrV_1$, since for each fixed $x$, the kernel of $\frT_\alpha$ decays like $\xi^{-2-2\alpha}$ as $\xi\to\infty$.

From definition of $F_{1,-2\alpha}$ in Lemma \ref{Lemma: AuxiliaryFunctionF1}, one can show that $F_{1,-2\alpha}'(1/t) = t^{2+2\alpha} F_{1,-2\alpha}'(t), F_{1,-2\alpha}(0)=F_{1,-2\alpha}'(0)=0, \lim_{x\to+\infty} F_{1,-2\alpha}(x) = \frac{2}{1-2\alpha} , F_{1,-2\alpha}'(t) >0 $ on $t\in(0,\infty)$.

Then we know that
 \begin{equation} 
\begin{cases}
\frT_\alpha(f)(0) = 0\\
\mathfrak{b}(f) := - U'(0)= 4c_{0,\alpha} \int_{\xi>0} \frac{f(\xi)}{|\xi|^{2\alpha}}\, \idiff \xi \\
 \mathfrak{b}(f) = \lim_{x\to\infty} \frT_\alpha(f)(x) \\
 \frac{V(x)}{x} = \frT_\alpha(f)(x) + 1
     \end{cases} 
   \ . \end{equation}
    
    Let us define operator $\frR_\alpha$ by
$$ \frR_\alpha(f)(x) = \exp\rbracket{-\int_0^x \frac{\frT_\alpha(f)(y)}{y (1+\frT_\alpha(f)(y))} \, dy } \ .$$
    
    We now aim to study the fixed-point problem $f= \frR_\alpha(f), f\in\scrV_1$. As the core idea of this paper, the following proposition explains how a fixed-point of $\frR_\alpha$ is related to a solution to \eqref{Equation: SelfSimilarProfileOmegaHP}.

\begin{proposition}\label{Proposition: relation1Dmodel2DSQGHP} For any $\alpha\in(0,\frac12)$, if $f\in\scrV_1$ is a fixed-point of $\frR_\alpha$, i.e., $\frR_\alpha(f)=f$, then $f(+\infty)=0, \mathfrak{b}(f)<\infty$, and $\Theta=xf$ is a solution to \eqref{Equation: SelfSimilarProfileOmegaHP} with 
\begin{equation}\label{Equation: c_l HP} c_\theta=1, c_\ell  = 1- U'(0) = 1 + 4c_{0,\alpha} \int_{\xi>0} \frac{f(\xi)}{|\xi|^{2\alpha}}\, \idiff \xi \ .
\end{equation}
Moreover, let 
$ \lambda = \rbracket{2\alpha c_\ell -1}^{\frac{1}{1-2\alpha}} $
then $(\Theta_\lambda(x),\widetilde{c_\ell },\widetilde{c_\theta}) = (\lambda^{-1}\Theta(\lambda x), \lambda^{2\alpha-1} c_\ell , \lambda^{2\alpha-1} c_\theta) $ is a solution to \eqref{Equation: SelfSimilarProfileOmegaHP} with $\widetilde{c_\theta} - 2\alpha \widetilde{c_\ell } = -1$.

Conversely, if $(\Theta,\widetilde{c_\ell },\widetilde{c_\theta})$ is a solution to \eqref{Equation: SelfSimilarProfileOmegaHP}  with $\widetilde{c_\theta} - 2\alpha \widetilde{c_\ell } = -1$ such that $f=\Theta(x)/x$ is an even function of $x$, $f(x)\geq0$ and $f'(x)\leq0$ for $x\in[0,\infty)$, $f(0)=1$, $\lim_{x\to+\infty} f(x)=0$, then let $\lambda = \widetilde{c_\theta}^{\frac{1}{1-2\alpha}}$ and we renormalize $(f_\lambda(x), c_\ell , c_\theta) = (f(\lambda x), \widetilde{c_\ell }/ \widetilde{c_\theta}, 1)$, then $f_\lambda$ is a fixed-point of $\frR_\alpha$, and $c_\ell $ is related to $f_\lambda$ as in \eqref{Equation: c_l HP}.
\end{proposition}
\begin{proof} The first statement follows directly from the construction of $\frR_\alpha$. The claim that $\mathfrak{b}(f)<\infty$ provided $f=\frR_\alpha(f)$ will be proved in Lemma \ref{Lemma: bcbound} below. The formulas for $c_\theta$ and $c_\ell $ are formed by construction. Conversely, if $\Theta=xf$ is a solution to equations \eqref{Equation: SelfSimilarProfileOmegaHP} then we can utilize \eqref{Equation: scaling invariant HP} to rescale $\Theta$  so that \eqref{Equation: c_l HP} holds.
\end{proof}


\subsubsection{ Roadmap for proofs }
The remainder of this section is devoted to proving the existence of fixed points of $\frR_\alpha$ on $\scrV_1$. Similar to the procedure in the previous section, we want that 
\begin{enumerate}
\item $\scrV_1$ is nonempty, convex, and closed in the underlying Banach topology;
\item $\scrV_1$ is a compact set in this Banach space; 
\item $\frak{R}_\alpha$ maps $\scrV_1$ continuously into itself in the same topology;
\end{enumerate}
 We note that (1) is automatically satisfied by the design of $\scrV_1$. To establish (2) and (3), it is crucial to observe that the intermediate linear map $\frT_\alpha$ preserves monotonicity and square-root-convexity/concavity on $[0,\infty)$, which can be proven by applying integration by parts to the formula of $\frT_\alpha$. This monotonicity and convexity preserving property of $\frT_\alpha$ then passes on to the non-linear map $\frR_\alpha$  through some straightforward calculations of derivatives, which provide powerful controls on $\frT_\alpha$. Moreover, the monotonicity and convexity properties lead to the continuity of $\frT_\alpha:\scrV_1\to\scrV_1$ and the compactness of $\scrV_1$ in the $L^\infty$-topology. Finally, with all these ingredients in hand, we can apply the Schauder fixed-point theorem on $\frR_\alpha:\scrV_1\to\scrV_1$ to conclude the proof.
    

\subsubsection{Properties of $\frT_\alpha$ and $\frR_\alpha$ } We now turn to study the intermediate maps $\frT_\alpha$ and $\frR_\alpha$. As an important observation in our fixed-point method, they preserve the monotonicity and convexity/concavity of functions in $\scrV_1$.
\begin{proposition}
For any $\alpha\in(0,\frac12)$, $\frT_\alpha(f)(\sqrt{x})$ is concave on $[0,\infty)$ for $f\in\scrV_1$.
\end{proposition}
\begin{proof}
It suffices to show that $\rbracket{\frac{\frT_\alpha(f)'(x)}{ x}}' \leq 0$ on $(0,\infty)$. Indeed, we can compute that
\begin{equation}
\begin{split} \label{Equation: convexity}
& {\frac{\frT_\alpha(f)'(x)}{ x}} \\
 = &  -2c_{0,\alpha} \int_{\xi>0} \frac{f'(\xi)}{\xi} \frac{\xi^{1-2\alpha}}{x} F_{1,-2\alpha}'(x/\xi)\, \idiff \xi \\
 = &- 2c_{0,\alpha} \int_{\xi>0} \frac{f'(\xi)}{\xi} \rbracket{ \frac{\xi^{1-2\alpha}}{x} F_{1,-2\alpha}'(x/\xi) + \frac{8\alpha(1+\alpha)}{3}\xi^{-2\alpha} }\, \idiff \xi \\
 &+ \frac{8\alpha(1+\alpha)}{3} c_{0,\alpha}'\int_{\xi>0} \frac{f'(\xi)}{\xi^{1+2\alpha}} \, \idiff \xi \\ 
 = & -2c_{0,\alpha} \int_{\xi>0} \rbracket{ \frac{f'(\xi)}{\xi} }' \xi^{1-2\alpha} F_{2,-2\alpha}(x/\xi) \, \idiff \xi + \mathfrak{c}(f) \ ,
\end{split} 
\end{equation}
where  constant $\mathfrak{c}(f):=  \frac{8(1+\alpha)}{3} c_{0,\alpha} \int_{\xi>0} \frac{f'(\xi)}{\xi^{1+2\alpha}} \, \idiff \xi $ has the following properties:
\begin{equation}
 \mathfrak{c}(f) = -\frac{8(1+2\alpha)(1+\alpha)}{3}c_{0,\alpha} \int_{\xi>0} \frac{1-f(\xi)}{\xi^{2+2\alpha}}\, \idiff \xi = \lim_{x\to 0} {\frac{\frT_\alpha(f)'(x)}{ x}} \ ,
\end{equation}
and the definition and properties of the auxiliary function $F_{2,-2\alpha}$ are left as  Lemma \ref{Lemma: Auxiliary Function2} in Appendix \ref{Section: Auxiliary functions}. In particular, we have $F_{2,-2\alpha}'(x/\xi)\geq0$. As a consequence, we can compute that 
$$\rbracket{\frac{\frT_\alpha(f)'(x)}{ x}}'   = -2c_{0,\alpha} \int_{\xi>0} \rbracket{ \frac{f'(\xi)}{\xi} }' \xi^{-2\alpha} F_{2,-2\alpha}'(x/\xi) \, \idiff \xi \leq 0 \ , $$
which completes the proof that $\frT_\alpha(f)(\sqrt{x})$ is concave on $(0,\infty)$.\end{proof}

The concavity of $\frT_\alpha(f)(\sqrt{x})$ gives us the following corollary:
\begin{corollary}\label{Corollary: Concavity_T} For $f\in\scrV_1$, we have
\begin{equation*}
\begin{split}
&  \frac{\frT_\alpha(f)(x)}{x^2} \geq \frac{\frT_\alpha(f)'(x)}{2x} \ , \\
& \rbracket{\frac{\frT_\alpha(f)(x)}{x^2}}' \leq 0  \ , \\
&  \frac{\frT_\alpha(f)(x)}{x^2} \leq \lim\limits_{x\to 0} \frac{\frT_\alpha(f)'(x)}{2x}= \frac{\frak{c}(f)}{2} \ , \\
& \frT_\alpha(f)(x) \leq \min\rbracket{ \frak{b}(f), \frac{\frak{c}(f)}{2}x^2} \ .
\end{split} 
\end{equation*}
\end{corollary}

\begin{proposition}\label{Proposition: RalphaConvexity} For any $\alpha\in(0,\frac12)$, $\frR_\alpha(f)(x)$ is non-increasing and $\frR_\alpha(f)(\sqrt{x})$ is convex on $(0,\infty)$ for any $f\in\scrV_1$.
\end{proposition}
\begin{proof}
The fact that $\frR_\alpha(f)(x)$ is non-increasing on $(0,\infty)$ can be shown directly by
\begin{equation}\label{Equation: Derivative R}
 \frR_\alpha(f)'(x) =  -  \frac{\frR_\alpha(f)(x) \frT_\alpha(f)(x)}{x (1+\frT_\alpha(f)(x))}  \leq 0 \ .
 \end{equation}
 Moreover, from
 $$  \frT_\alpha(f)'(x) = -2c_{0,\alpha} \int_{\xi>0} f'(\xi) \xi^{-2\alpha} F_{1,-2\alpha}'(x/\xi) \, \idiff \xi \ .
$$ 
one can conclude that $\frT_\alpha(f), \frT_\alpha'(f)\geq0$ on $[0,\infty)$ for any $f\in\scrV_1$. 

We observe that 
$$ \rbracket{\frac{\frT_\alpha(f)(x)}{x^2}}' \leq 0, \ \ \ \ \    \rbracket{\frac{\frR_\alpha(f)(x)}{1+\frT_\alpha(f)(x)} }' \leq 0 \ .  $$
As a consequence, we obtain
$$ \rbracket{\frac{\frR_\alpha(f)'(x)}{x}}' =  -\rbracket{\frac{\frT_\alpha(f)(x)}{x^2}}' \frac{\frR_\alpha(f)(x)}{1+\frT_\alpha(f)(x)} - {\frac{\frT_\alpha(f)(x)}{x^2}} \rbracket{\frac{\frR_\alpha(f)(x)}{1+\frT_\alpha(f)(x)} }' \geq0 \ ,
$$
and hence  $\frR_\alpha(f)(\sqrt{x})$ is convex on $(0,\infty)$.
\end{proof}

\begin{lemma}\label{Lemma: bcbound} 
For any $f\in\scrV_1$, the following inequality holds for constant $\mathfrak{b}(\Omega)$ and $\mathfrak{c}(f)$:
 \begin{equation*}
\begin{split}
& b_\alpha'' \geq \mathfrak{b}(f) \geq b_\alpha', \ \  c_{0,\alpha}'' \geq \mathfrak{c}(f) \geq  \frac{8(1+\alpha)}{3(1-2\alpha)}c_{0,\alpha}  t_0^{-\frac32-\alpha}  \ , \\
& b_\alpha'':= 4 c_{0,\alpha} t_0^{1-2\alpha} \rbracket{\frac{1}{1-2\alpha} + \frac{1}{\delta_u+2\alpha-1}} \ , \\
& b_\alpha':=2c_{0,\alpha} t_0^{1-2\alpha} \frac{\Gamma(\frac12-\alpha)\Gamma(\alpha+\delta_l-\frac12)}{\Gamma(\delta_l)} \ , \\
& c_{0,\alpha}'' := c_{0,\alpha} t_0^{-1-2\alpha} \frac{8(1+\alpha)\Gamma(\frac12-\alpha)\Gamma(\delta_l+\alpha+\frac12)}{3 \Gamma(\delta_l)} \ .
 \end{split} 
 \end{equation*}
 \end{lemma}
 \begin{proof}
This can be shown directly by 
 \begin{equation*}
\begin{split}
  \mathfrak{b}(f) &= 4 c_{0,\alpha} \int_{\xi>0} \frac{f(\xi)}{|\xi|^{2\alpha}}\, \idiff \xi \geq  4 c_{0,\alpha} \int_{\xi>0} \rbracket{1+t_0^{-2}\xi^2}^{-\delta_l}\xi^{-2\alpha} \, \idiff \xi = b_{\alpha}' \ , \\
  \mathfrak{b}(f) & \leq 4 c_{0,\alpha} \rbracket{ \int_0^{t_0} \xi^{-2\alpha}\, \idiff \xi + \int_{t_0}^\infty t_0^{\delta_u} x^{-2\alpha-\delta_u}\, \idiff \xi } = b_\alpha'' \ , \\
    \mathfrak{c}(f) & = \frac{8(1+2\alpha)(1+\alpha)}{3}c_{0,\alpha} \int_{\xi>0} \frac{1-f(\xi)}{\xi^{2+2\alpha}}\, \idiff \xi  \\
& \leq  \frac{8(1+2\alpha)(1+\alpha)}{3}c_{0,\alpha} \int_{\xi>0} \frac{1-(1+t_0^{-2}\xi^2)^{-\delta_l}}{\xi^{2+2\alpha}}\, \idiff \xi = c_{0,\alpha}'' \ .
 \end{split} 
\end{equation*}

Moreover, $f(\sqrt{x})$ being convex and $f'(t_0)\leq -t_0^{\alpha-\frac32}$ indicates that
$$ f(x) \leq \max\rbracket{1-\frac{1}{2} t_0^{\alpha-\frac52} x^2, 1-\frac{1}{2} t_0^{\alpha-\frac12} } \ ,$$
and thus we get
 \begin{equation*}
\begin{split}
\mathfrak{c}(f) & = \frac{8(1+2\alpha)(1+\alpha)}{3}c_{0,\alpha} \int_{\xi>0} \frac{1-f(\xi)}{\xi^{2+2\alpha}}\, \idiff \xi  \\
& \geq  \frac{8(1+2\alpha)(1+\alpha)}{3}c_{0,\alpha}  \rbracket{ \int_{0}^{t_0} \frac{ t_0^{\alpha-\frac52}}{2 \xi^{2\alpha}}\idiff \xi + \int_{t_0}^\infty \frac{ t_0^{\alpha-\frac12}}{2 \xi^{2+2\alpha}}\idiff \xi }\\ & = \frac{8(1+\alpha)}{3(1-2\alpha)} c_{0,\alpha} t_0^{-\alpha-\frac32}  \ . \qedhere
 \end{split} 
\end{equation*}
\end{proof}

\begin{proposition}\label{Proposition: Lowerbound F2} On $\scrV_1$, operator $\frR_\alpha$ preserves the lower bound $f_l$.
\end{proposition}
\begin{proof}
Using $\frT_\alpha(f)(x)\leq \frac{\mathfrak{c}(f)}{2}x^2\leq \frac{c_{0,\alpha}''}{2}x^2$, we have
 \begin{equation*}
\begin{split}
\frR_\alpha(f)(x) & = \exp\rbracket{-\int_0^x \frac1y \rbracket{ 1-\frac{1}{1+\frT_\alpha(f)(y)} }  \idiff y }   \\
& \geq \exp\rbracket{-\int_0^x \frac1y \rbracket{ 1-\frac{1}{1+\frac{c_{0,\alpha}''}{2}y^2} }  \idiff y } = \rbracket{1+\frac{c_{0,\alpha}''}{2}x^2}^{-\frac12}.
 \end{split} 
\end{equation*}
To show that $\frR_\alpha(f)(x)\geq f_l(x)$, it suffices to show that for all $x\geq 0$,
$$ 1 + \frac{c_{0,\alpha}''}{2}x^2 \leq \rbracket{1 + t_0^{-2}x^2}^{2\delta_l}  \ , $$
which is equivalent to the following condition:
$$
c_{0,\alpha}'' \leq 4\delta_l t_0^{-2} \text{ and } \delta_l\geq\frac12 \ ,
$$
which is guaranteed by \eqref{Equation: deltal}.
\end{proof}

\begin{proposition}\label{Proposition: Derivativebound F2}  On $\scrV_1$, operator $ \frR_\alpha$ preserves the derivative bound $f'(t_0)\leq -t_0^{\alpha-\frac32}$.
\end{proposition}
\begin{proof}
By convexity of $f(\sqrt{x})$, we have $\frac{f'(x)}{x}\leq \frac{f'(t_0)}{t_0}\leq -{t_0^{\alpha-\frac52}}$ for all $0\leq x<t_0$. Then, we obtain
\begin{equation}\label{Equation: Bound Tp}
\begin{split}
  \frT_\alpha(f)'(t_0)  & = -2c_{0,\alpha} \int_{\xi>0} f'(\xi) \xi^{-2\alpha} F_{1,-2\alpha}'(t_0/\xi) \, \idiff \xi\\
  &\geq  -2c_{0,\alpha} \int_0^{t_0}  f'(\xi) \xi^{-2\alpha} F_{1,-2\alpha}'(t_0/\xi) \, \idiff \xi \\
  & =  -2c_{0,\alpha} t_0^{-2-2\alpha} \int_0^{t_0}  f'(\xi) \xi^{2} F_{1,-2\alpha}'(\xi/t_0) \, \idiff \xi  \\
  &=  -2c_{0,\alpha}  t_0^{1-2\alpha} \int_0^1  f'(t_0\xi) \xi^{2} F_{1,-2\alpha}'(\xi) \, \idiff \xi  \\
   & \geq   2c_{0,\alpha} t_0^{\alpha-\frac32} t_0^{1-2\alpha} \int_0^1 \xi^{3} F_{1,-2\alpha}'(\xi) \, \idiff \xi    \\
 &= \frac{2^{1 - 2\alpha}\alpha(1 + \alpha)  2c_{0,\alpha} t_0^{-\alpha-\frac12} }{(1 - \alpha)(1-2\alpha)(2-\alpha)} \ , 
  \end{split}
\end{equation}
  
  \begin{equation}\label{Equation: Bound T}
\begin{split}
\frT_\alpha(f)(t_0)  & =  -2c_{0,\alpha} \int_{\xi>0} f'(\xi) \xi^{1-2\alpha} F_{1,-2\alpha}'(t_0/\xi) \, \idiff \xi \\
   &\geq  -2c_{0,\alpha}\int_0^{t_0}  f'(\xi) \xi^{1-2\alpha} F_{1,-2\alpha}'(t/\xi) \, \idiff \xi \\
  & =   -2c_{0,\alpha} t_0^{-2-2\alpha} \int_0^{t_0}  f'(\xi) \xi^{3} F_{1,-2\alpha}'(\xi/t_0) \, \idiff \xi \\
  &=   -2c_{0,\alpha} t_0^{\alpha-\frac32}t_0^{2-2\alpha} \int_0^1  f'(t_0\xi) \xi^{3} F_{1,-2\alpha}'(\xi) \, \idiff \xi  \\
 & \geq  2c_{0,\alpha} t_0^{\alpha-\frac32} t_0^{2-2\alpha} \int_0^1 \xi^{4} F_{1,-2\alpha}'(\xi) \, \idiff \xi    = c_{0,\alpha}''' t_0^{\frac12-\alpha} \ , \\
 \frT_\alpha(f)(t_0) &  \leq \min\rbracket{\frac{\mathfrak c(f)}{2} t_0^2, \mathfrak{b}(f)} \leq \min\rbracket{\frac{c_{0,\alpha}''}{2}t_0^2, b_\alpha'' } \ .
 \end{split} 
\end{equation}

From Proposition  \ref{Proposition: Lowerbound F2}, we know that 
$$ \frR_\alpha(f)(t_0) \geq \rbracket{1+\frac{c_{0,\alpha}'' t_0^2}{2}}^{-\frac12} .$$
Then we can compute directly from \eqref{Equation: t0},  \eqref{Equation: deltal}, and  \eqref{Equation: Derivative R} that
 $$
\frac{ \frR_\alpha(f)'(t_0) }{-t_0^{\alpha-\frac32} }=  \frac{\frR_\alpha(f)(t_0) \frT_\alpha(f)(t_0)}{t_0^{\alpha-\frac32}(1+\frT_\alpha(f)(t_0))}   \geq \frac{c_{0,\alpha}'''  t_0^{2-2\alpha}}{1+c_{0,\alpha}''' t_0^{\frac12-\alpha}}  \rbracket{1+\frac{c_{0,\alpha}'' t_0^2}{2}}^{-\frac12}   \geq 1 \ . \qedhere
$$
\end{proof}

\begin{proposition}\label{Proposition: Upperbound F2}  On $\scrV_1$, operator $ \frR_\alpha$ preserves the upper bound $f_u$.
\end{proposition}
\begin{proof} From equations \eqref{Equation: Bound Tp} and \eqref{Equation: Bound T}, we have for any $y\geq t_0$,
\begin{equation*}
 \frT_\alpha(f)(y)\geq c_{0,\alpha}''' t_0^{\frac12-\alpha}\  \Rightarrow \  1-\frac{1}{1+\frT_\alpha(f)(y)} \geq \delta_u =\frac{c_{0,\alpha}''' t_0^{\frac12-\alpha}}{1+c_{0,\alpha}''' t_0^{\frac12-\alpha}} \ .
\end{equation*}
Consequently, for any $x>t_0$, we have
 \begin{equation*}
\begin{split}
\frR_\alpha(f)(x) & = \exp\rbracket{-\int_0^x \frac1y \rbracket{ 1-\frac{1}{1+\frT_\alpha(f)(y)} }  \, dy }   \\
& \leq  \exp\rbracket{-\int_{t_0}^x \frac1y \rbracket{ 1-\frac{1}{1+\frT_\alpha(f)(y)} }  \, dy }   \\
& \leq \exp\rbracket{- \delta_u \int_{t_0}^x \frac1y  \, dy } = \rbracket{\frac{x}{t}}^{-\delta_u} = f_u \ .  \qedhere
 \end{split} 
\end{equation*}
\end{proof}

Propositions \ref{Proposition: RalphaConvexity}, \ref{Proposition: Lowerbound F2}, \ref{Proposition: Derivativebound F2}, and \ref{Proposition: Upperbound F2} directly give us the following corollary.
\begin{corollary} For any $\alpha\in(0,\frac12)$,  $\frR_\alpha$ maps $\scrV_1$ to itself.
\end{corollary}

Next, we show that $\frR_\alpha$ is continuous on $\scrV_1$ in the $L^\infty$-topology.
\begin{proposition}\label{Proposition: ContinuityRaHP} For any $\alpha\in(0,\frac12)$,  $\frR_\alpha: \scrV_1 \to \scrV_1$ is continuous in $L^\infty$-norm.
\end{proposition}
\begin{proof}
Let us fix any $f_0 \in \scrV_1 $.  For any $f\in\scrV_1$, Proposition \ref{Proposition: Upperbound F2} implies that $ \frR_\alpha(f)(x)\leq f_u(x) \leq x^{-\delta_u} $. Let $\norm{f - f_0}_{L^\infty}\leq \delta$ for some small $\delta$. Then we have
\begin{align*}
& \abs{\frT_\alpha(f)(x) - \frT_\alpha(f_0)(x)} \\
 = & c_{0,\alpha}' \abs{ \int_{\xi>0} \rbracket{f(\xi) - f_0(\xi)} \rbracket{\frac{\xi  }{x|x-\xi|^{2\alpha}}-\frac{\xi  }{x|x+\xi|^{2\alpha}} - 4\alpha\frac{1}{|\xi|^{2\alpha}}} \, \idiff \xi } \\
 \leq &  \delta c_{0,\alpha}' \int_{\xi>0} \abs{ {\frac{\xi  }{x|x-\xi|^{2\alpha}}-\frac{\xi  }{x|x+\xi|^{2\alpha}} - 4\alpha\frac{1}{|\xi|^{2\alpha}}}} \, \idiff \xi \\
 = &  \delta x^{1-2\alpha}  2c_{0,\alpha}\int_{\xi>0} \abs{ {\frac{\xi  }{|1-\xi|^{2\alpha}}-\frac{\xi  }{|1+\xi|^{2\alpha}} - 4\alpha\frac{1}{|\xi|^{2\alpha}}}} \, \idiff \xi  \\
\leq &  \delta x^{1-2\alpha}  2c_{0,\alpha} \int_{0}^2  {\frac{\xi  }{|1-\xi|^{2\alpha}}+\frac{\xi  }{|1+\xi|^{2\alpha}} + 4\alpha\frac{1}{|\xi|^{2\alpha}}} \, \idiff \xi \\
& + \delta x^{1-2\alpha}  2c_{0,\alpha} \int_{2}^\infty \xi^{1-2\alpha}  \rbracket{\rbracket{1-\xi^{-1}}^{-2\alpha} - \rbracket{1+\xi^{-1}}^{-2\alpha} - 4\alpha\xi^{-1}} \, \idiff \xi    \\
\leq & \delta x^{1-2\alpha}  2c_{0,\alpha} \Bigg( \frac{5+3^{1-2\alpha}(1-4\alpha) - 4^{1-\alpha}\alpha(4^\alpha+4\alpha-4) }{2(1-\alpha)(1-2\alpha)} \\
& \quad\quad \quad\quad\quad\quad + \frac{\alpha(\alpha+1)(2\alpha+1)2^{2\alpha+4}}{3} \int_2^\infty \xi^{-2-2\alpha}\, \idiff \xi \Bigg) \\
\leq & \delta x^{1-2\alpha}  2c_{0,\alpha} \rbracket{ \frac{5}{1-2\alpha}   } \ .
\end{align*}
Therefore, we have
\begin{align*}
\abs{ \frR_\alpha(f)(x)-\frR_\alpha(f_0)(x) } & \leq \int_0^x \frac1y \abs{  \frac{\frT_\alpha(f)(y)}{1+\frT_\alpha(f)(y)} -  \frac{\frT_\alpha(f_0)(y)}{1+\frT_\alpha(f_0)(y)} }  \, dy  \\
& = \int_0^x \frac1y   \frac{\abs{\frT_\alpha(f)(y)-\frT_\alpha(f_0)(y)}}{(1+\frT_\alpha(f)(y))(1+\frT_\alpha(f_0)(y))}   \, dy  \\
& \leq \delta  2c_{0,\alpha} \rbracket{ \frac{5}{1-2\alpha}  }
 \int_0^x y^{-2\alpha}  \, dy  \\
 & \leq \frac{10 c_{0,\alpha}}{(1-2\alpha)^2} \delta x^{1-2\alpha} \ .
\end{align*}
Thus, for any $\varepsilon>0$, let $M = \rbracket{\frac{2t_0}{\varepsilon}}^{\frac{1}{\delta_u}}>1$ such that 
$ \frR_\alpha(f)(x)\leq \frac{\varepsilon}{2} $ for all $x\geq M$. We take sufficiently small $\delta$ so that $ \frac{ 10c_{0,\alpha}}{(1-2\alpha)^2} \delta M^{1-2\alpha} \leq \varepsilon$ and thus
$$ \norm{\frR_\alpha(f) - \frR_\alpha(f_0)}_{L^\infty} \leq \varepsilon \ ,$$
which finishes the proof that operator $\frR_\alpha$ is continuous on $\scrV_1$ in $L^\infty$-norm.
\end{proof}


\subsubsection{Existence of fixed-point $f_*$ of $\frR_\alpha$}
\begin{lemma}\label{Lemma: CompactSetV1HP}  The function set $\scrV_1$ is a compact subset of Banach space $\scrV_0$.
\end{lemma}
\begin{proof} For any $f\in\scrV_1$, by convexity of $f(\sqrt{x})$, we have
$$ -\frac{f'(x)}{2x} \leq \frac{1-f(x)}{x^2}\leq \frac{1-f_l(x)}{x^2} \leq \min(1, x^{-2}) \ ,$$
and thus $\abs{f'(x)}\leq \min(2x,2x^{-1})\leq 2$. Therefore, functions in $\scrV_1$ are uniformly bounded and equicontinuous, so we will use the Arzela-Ascoli theorem. More specifically, we consider any sequence $\{f_i\}_{i=1}^\infty\subseteq\scrV_1$ and perform the following steps of iteration over $m$. Start with $m=0$, take $n_{0,k}=k$ for all $k=1,\dots,n$. For any positive integer $m\geq1$, let $\varepsilon_m = 2^{-m}$ and $L_m =\rbracket{\frac{2t_0}{\varepsilon_m}}^{\frac{1}{\delta_u}}$. Then for any $x\geq L_m$, one has $\sup\limits_{i\geq1}f_i(x)\leq \frac{\varepsilon_m}{2}$. Moreover, since $\sup\limits_{i\geq 1}\sup\limits_{0\leq x\leq L_m} \abs{f_i'(x)}\leq 2$, the Arzela-Ascoli theorem states that there exists subsequence $\{f_{n_{m,k}}\}_{k=1}^\infty\subseteq \{f_{n_{m-1,k}}\}_{k=1}^\infty$ so that $\sup\limits_{i,j\geq1}\norm{f_{n_{m,i}}-f_{n_{m,j}} }_{L^\infty}\leq \varepsilon_m$. Then we take the diagonal terms and get $\{f_{n_{m,m}}\}_{m=1}^\infty$ which is a Cauchy sequence in $L^\infty$ norm, and this finishes the proof that $\scrV_1$ is compact.
\end{proof}

We are now ready to prove the existence of fixed points of $\frR_\alpha$ for any $\alpha\in(0,\frac12)$ using the Schauder fixed-point theorem.
\begin{theorem} For any $\alpha\in(0,\frac12)$, the map $\frR_\alpha:\scrV_1\to\scrV_1$ has a fixed point.
\end{theorem}

\begin{proof}
By Proposition \ref{Proposition: ContinuityRaHP} and Lemma \ref{Lemma: CompactSetV1HP}, $\scrV_1$ is convex, closed and compact in the $L^\infty$-norm, and $\frR_\alpha$ continuously maps $\scrV_1$ into itself. The Schauder fixed-point theorem implies that $\frR_\alpha$ has a fixed point in $\scrV_1$. 
\end{proof}

Here we verify that the self-similar profile we found with unbounded support indeed corresponds to a focusing-type self-similar blow-up.
\subsubsection{Property of fixed-point $f_*$ of $\frR_\alpha$}
\begin{proposition} Let $f_*\in\scrV_1$ denote a fixed-point of $\frR_\alpha$. Then we have $2\alpha c_\ell -1>0$ and thus $f_*$ corresponds to a focusing-type blow-up.
\end{proposition}
\begin{proof} 
Because $f_*\in\scrV_1$, Corollary \ref{Corollary: Concavity_T} implies that $\frT_\alpha(f_*)(x)\leq \min(\mathfrak{b}(f_*), \frac{\mathfrak{c}(f_*)}{2}x^2)$, so for $x>1$, one has
 \begin{equation*}
\begin{split}
& f_*(x) = \frR_\alpha(f_*)(x) \\
 = & \exp\rbracket{-\int_0^x \frac1y \rbracket{ 1-\frac{1}{1+\frT_\alpha(f_*)(y)} }  \, dy }   \\
 \geq & \exp\rbracket{-\int_0^1 \frac1y \rbracket{ 1-\frac{1}{1+\frac{\mathfrak{c}(f_*)}{2}y^2} }  \, dy -\int_1^x \frac1y \rbracket{ 1-\frac{1}{\mathfrak{b}(f_*)} }  \, dy  } \\
 = & \rbracket{1+\frac12 \mathfrak{c}(f_*)}^{-\frac12} x^{-\frac{\mathfrak{b}(f_*)}{1+\mathfrak{b}(f_*)}} \ .
 \end{split} 
\end{equation*}
Consequently, the boundedness of 
\begin{align*}
\mathfrak{b}(f_*) & = 4 c_{0,\alpha} \int_{\xi>0} f_*(\xi) \xi^{-2\alpha}\, \idiff \xi \\
& \geq 4 c_{0,\alpha}  \rbracket{1+\frac12 \mathfrak{c}(f_*)}^{-\frac12} \int_{\xi>1} \xi^{-\frac{ \mathfrak{b}(f_*)}{1+\mathfrak{b}(f_*)}-2\alpha}\, \idiff \xi 
\ , \end{align*}
implies that 
$\frac{ \mathfrak{b}(f_*)}{1+\mathfrak{b}(f_*)}+2\alpha>1$, which implies that $ \mathfrak{b}(f_*) > \frac{1}{2\alpha}-1$, and hence $c_\ell  = 1 + \mathfrak{b}(f_*)>\frac1{2\alpha}$.
\end{proof}

To obtain higher regularity of $f$, we will show in the following lemma on the regularity for operator $\frT_\alpha$.

\begin{lemma}\label{Lemma: RegMapTaHP} Given $f\in\scrV_1$, if $f'\in H^s(\R)$ for some $s \geq 1$, then $\frT_\alpha(f)'\in H^{s+1-2\alpha-2\eps}(\R)$ for any $0<\eps<1/2-\alpha$.
\end{lemma}

\begin{proof}
 According to Corollary\ref{Corollary: Concavity_T}, 
    \[
    0\leq\mathfrak{T}_{\alpha}(f)^{\prime}(x)\leq\frac{2\mathfrak{T}_{\alpha}(f)(x)}{x}\leq \min\bigg(\frac{2\mathfrak{b}(f)}{x},\mathfrak{c}(f)x\bigg).
    \]
    In particular, this implies
    \[
    \mathfrak{T}_{\alpha}(f)^{\prime}(x)\in L^2(\R),
    \]
   We recall that\[
    \frT_\alpha(f)'=\bigg(\frac{U}{x}\bigg)'\sim\bigg(\frac{(-\Delta) ^{\alpha-1/2}(xf)}{x}\bigg)',
    \]
    up to a constant factor. For simplicity we suppress this harmless constant factor in the estimates below.

We next derive a convenient representation for the Fourier transform of $\frT_\alpha(f)'$.
Set
\[
h(x):=\frac{(-\Delta)^{\alpha-1/2}(xf)(x)}{x},
\qquad\text{so that}\qquad
\frT_\alpha(f)' \sim h'(x).
\]
Since $x h(x)=(-\Delta)^{\alpha-1/2}(xf)(x)$, taking the Fourier transform gives
\[
i\partial_\xi \calF(h)(\xi)
=
\calF\!\left((-\Delta)^{\alpha-1/2}(xf)\right)(\xi)
=
(1+\xi^2)^{\alpha-1/2}\calF(xf)(\xi).
\]
Moreover,
\[
\calF(xf)(\xi)=i\,\partial_\xi \calF(f)(\xi),
\qquad
\calF(f')(\xi)=i\xi\,\calF(f)(\xi),
\]
hence
\[
\calF(xf)(\xi)
=
\left(\frac{\calF(f')(\xi)}{\xi}\right)'
\]
for $\xi\neq 0$. Therefore,
\[
i\partial_\xi \calF(h)(\xi)
=
(1+\xi^2)^{\alpha-1/2}
\left(\frac{\calF(f')(\xi)}{\xi}\right)'.
\]
Since
\[
\calF(\frT_\alpha(f)')(\xi)\sim i\xi\,\calF(h)(\xi),
\]
we obtain
\[
\left(\frac{\calF(\frT_\alpha(f)')(\xi)}{\xi}\right)'
\sim
(1+\xi^2)^{\alpha-1/2}
\left(\frac{\calF(f')(\xi)}{\xi}\right)'
\]
for $\xi\neq 0$. Now fix $\xi>1$. Integrating the above identity on $[\xi,R]$ yields
\[
\frac{\calF(\frT_\alpha(f)')(R)}{R}
-
\frac{\calF(\frT_\alpha(f)')(\xi)}{\xi}
\sim
\int_\xi^R
(1+\eta^2)^{\alpha-1/2}
\left(\frac{\calF(f')(\eta)}{\eta}\right)'
\,\mathrm d\eta.
\]
Since $\calF(\frT_\alpha(f)')\in L^2(\R)$, we have
\[
\calF(\frT_\alpha(f)')(R)\to 0
\qquad\text{along a subsequence }R\to+\infty,
\]
and hence
\[
\calF(\frT_\alpha(f)')(\xi)
\sim
\xi
\left(
-\int_\xi^{+\infty}
(1+\eta^2)^{\alpha-1/2}
\left(\frac{\calF(f')(\eta)}{\eta}\right)'
\,\mathrm d\eta
\right).
\]

    Since
    \[
    \|f'\|_{H^s(\R)}^2=\int_\R (1+\xi^2)^{s} |\calF(f')(\xi)|^2 \idiff\xi<+\infty,
    \]
    for $\xi\geq 1$ we now integrate by parts in $\eta$ and estimate,
    \[
    \begin{aligned}
        &\bigg|(1+\xi^2)^{s/2+1/2-\alpha-\eps}\calF\bigg(\frT_\alpha(f)'\bigg)(\xi)\bigg|\\
        \sim & \bigg|(1+\xi^2)^{s/2+1/2-\alpha-\eps}
\xi
\left(
-\int_\xi^{+\infty}
(1+\eta^2)^{\alpha-1/2}
\left(\frac{\calF(f')(\eta)}{\eta}\right)'
\,\mathrm d\eta
\right)\bigg|
         \\ 
        \leq &(1+\xi^2)^{s/2}|\calF(f')|+(1+\xi^2)^{s/2}\bigg|\int_\xi^{+\infty} \frac{\xi(1+\xi^2)^{1/2-\alpha-\eps}}{\eta}\bigg((1+\eta^2)^{\alpha-1/2}\bigg)'\calF(f')(\eta) \idiff \eta\bigg|\\
        \lesssim & (1+\xi^2)^{s/2}|\calF(f')|+\int_\xi^{+\infty}\bigg| \frac{\xi(1+\xi^2)^{s/2+1/2-\alpha-\eps}}{\eta(1+\eta^2)^{1/2-\alpha}}\frac{\calF(f')(\eta)}{\eta}\bigg| \idiff \eta\\
        \leq & (1+\xi^2)^{s/2}|\calF(f')| +\int_\xi^{+\infty}\bigg| (1+\eta^2)^{s/2-\eps}\frac{\calF(f')(\eta)}{\eta}\bigg| \idiff \eta.
        \end{aligned}\]
        Applying Cauchy–Schwarz to the integral term gives
        \[
        \begin{aligned}
        & \int_\xi^{+\infty}\bigg| (1+\eta^2)^{s/2-\eps}\frac{\calF(f')(\eta)}{\eta}\bigg| \idiff \eta\\
        \leq&  \|f'\|_{H^s(\R)}\left(\int_\xi^{+\infty}\left(\frac{(1+\eta^2)^{-\eps}}{\eta}\right)^2\idiff\eta\right)^{1/2}
        \lesssim \|f'\|_{H^s(\R)}\xi^{-1/2-\eps}.
    \end{aligned}
    \]
    Hence
\[
\left|
(1+\xi^2)^{s/2+1/2-\alpha-\eps}
\calF(\frT_\alpha(f)')(\xi)
\right|
\lesssim
(1+\xi^2)^{s/2}|\calF(f')(\xi)|
+
\|f'\|_{H^s(\R)}\,\xi^{-1/2-\eps}.
\]
Since the right-hand side belongs to $L^2([1,\infty))$, it follows that
\[
\big\|
(1+\xi^2)^{\frac s2+\frac12-\alpha-\eps}
\calF(\frT_\alpha(f)')
\big\|_{L^2([1,\infty))}
<+\infty.
\]
    On the low-frequency region, the fact that $\mathfrak{T}_{\alpha}(f)^{\prime}(x)\in L^2(\R)$ implies 
     \[
    \bigg\|(1+\xi^2)^{s/2+1/2-\alpha-\eps}\calF\bigg(\frT_\alpha(f)'\bigg)\bigg\|_{L^2([-2,2])}<+\infty.
    \]
    Recall that $f$ is an even function. Combining the high- and low-frequency estimates, we conclude
    \[
    \|\frT_\alpha(f)'\|_{H^{s+1-2\alpha-2\eps}(\R)}=\bigg\|(1+\xi^2)^{s/2+1/2-\alpha-\eps}\calF\bigg(\frT_\alpha(f)'\bigg)\bigg\|_{L^2(\R)}<+\infty,
    \]
    which is desired.
\end{proof}

Thus we have the following result on the smoothness for the fixed-point $f_*$ of operator $\frR_\alpha$.
\begin{proposition}\label{Proposition: SmoothnessProfileHP}
For any $\alpha\in(0,\frac12)$, let $f_*\in\scrV_1$ be a fixed-point of $\frR_\alpha$.  Then  $f_*$ is smooth.
\end{proposition}
\begin{proof}
Since $f_* \in \scrV_1$, it is continuous and satisfies $\frT_\alpha(f_*) \ge 0$. 
By the previously established upper bounds for 
$f_*$, $\frT_\alpha(f_*)$, $f_*'$, and $\frT_\alpha(f_*)'$, we have
\[
f_*,\, \frT_\alpha(f_*) \in L^\infty(\R),
\]
and
\[
f_*',\ \frT_\alpha(f_*)',\ \frac{f_*'}{x},\ \frac{\frT_\alpha(f_*)'}{x},\ 
\frac{\frT_\alpha(f_*)}{x},\ \frac{\frT_\alpha(f_*)}{x^2}
\in L^2(\R)\cap L^\infty(\R).
\]

Using the fact that $f_*$ is a fixed-point of $\frR_\alpha$,
\[
\frR_\alpha(f_*)' = f_*'
= -\frR_\alpha(f_*) \frac{\frT_\alpha(f_*)}{x\bigl(1+\frT_\alpha(f_*)\bigr)}
= -f_* \frac{\frT_\alpha(f_*)}{x\bigl(1+\frT_\alpha(f_*)\bigr)},
\]
we compute
\[
(f_*')'
= -f_*' \frac{\frT_\alpha(f_*)}{x(1+\frT_\alpha(f_*))}
- f_* \frac{\frT_\alpha(f_*)'}{x(1+\frT_\alpha(f_*))}
+ f_* \frac{\frT_\alpha(f_*)}{x^2(1+\frT_\alpha(f_*))}
+ f_* \frac{\frT_\alpha(f_*)\,\frT_\alpha(f_*)'}{x(1+\frT_\alpha(f_*))^2}.
\]
Each term on the right-hand side belongs to $L^2(\R)$ by the bounds listed above. 
Hence $(f_*')' \in L^2(\R)$, and therefore $f_*' \in H^1(\R)$.

Next, by Lemma \ref{Lemma: RegMapTaHP}, if $f' \in H^s(\R)$ for some $s \ge 1$, then
\[
\frT_\alpha(f)' \in H^{s+\frac{1-2\alpha}{2}}(\R).
\]
Moreover, by Hardy's inequality,
\[
\frac{\frT_\alpha(f)}{x} \in H^{s+\frac{1-2\alpha}{2}}(\R).
\]
Using the identity
\[
f_*' = - f_* \frac{\frT_\alpha(f_*)}{x\bigl(1+\frT_\alpha(f_*)\bigr)},
\]
together with standard product estimates in Sobolev spaces, we deduce that
\[
f_*' \in H^{s+\frac{1-2\alpha}{2}}(\R).
\]

Starting from $s=1$ and iterating this argument, we obtain successive gains of regularity. 
This bootstrapping procedure yields the smoothness of $f_*$.
\end{proof}

\section{Numerical Simulation for 1D reduction of SQG}\label{Section: Numerical Simulation}
For the one-dimensional nonlocal operators, we do not approximate the principal value integrals by naive pointwise quadrature at the singularity. Instead, on the computational grid (uniform in Section 4.1 and nonuniform in Section 4.2), we use a cubic-spline product-integration discretization: the singular kernel is integrated analytically against local spline basis functions. This produces finite discrete weights that incorporate the principal-value cancellation at the discrete level. In particular, diagonal and endpoint contributions are handled separately using closed-form limit formulas. We also exploit the odd-symmetry structure of the profile to reduce the full-line operator to a half-line computation with reflected contributions.

Near the singularity, no ad hoc $\varepsilon$-mollification is introduced in the operator itself; instead, the singular behavior is handled analytically through the matrix-entry formulas. In the implementation, special values at transition points (corresponding to normalized distances 0 and 1) are evaluated by explicit formulas, and different expressions are used in near-field and far-field regimes (including asymptotic expansions for large normalized distances) to reduce cancellation error.

For the velocity-level operator, we first approximate the derivative-level nonlocal operator and then reconstruct the velocity by cumulative summation, which is consistent with the structure of the reduced equation. The fixed-point iteration is stopped when the changes in the profile variable and scaling parameter are below the prescribed tolerance. In addition, numerical reliability is assessed by mesh-refinement and truncation-sensitivity checks (in particular when the profile is not compactly supported), and by verifying that the computed profiles satisfy the expected qualitative properties predicted by the analysis.

We emphasize that the existence and smoothness of the self-similar profiles are proved analytically in Sections 2–3. The numerical computations below are intended as qualitative consistency checks and visualization; they are not meant as a high-precision numerical study.  For the 1D nonlocal operators we assemble the dense $N\times N$ product-integration matrix (here $N=5\times 10^4$). The computations reported below were performed on a machine with 512 GB RAM; this approach is memory intensive and is used here for qualitative visualization rather than high-precision large-scale computation.


\subsection{Numerical Simulation for 1D reduction of gSQG on $\bbR^{2}$}
We use non-uniform mesh on $[0,5]$ with $N=5\times10^4$ grid points, specified by $x_i=5(i/N)^2$ for $i=0,1,\dots,N-1$. In other words, $\Delta x = 2\times10^{-9}$ at the origin and $\Delta x$ gradually increases to $2\times10^{-4}$ as $x$ approaches the endpoint $x=5$. The numerical method is to compute dense matrices with size $N\times N$ to approximate the operator $u_x  = \calP_{1,\alpha}\omega$ and to use the cumulative sum of $u_x$ to approximate the velocity $u=\calP_{0,\alpha}\omega$ and then $v=c_\ell x + u$. Thus, we use the iterative method in \eqref{Equation: fixedPointEqn} to solve the numerical approximation of a self-similar profile to \eqref{Equation: SelfSimilarProfileOmega}. We take $\varepsilon=10^{-7}$ and run the iteration until both changes in $w=-xf$ and $c_\ell $ are smaller than $\varepsilon$. 

\begin{figure}[!ht] 
\includegraphics[height=0.365\textwidth]{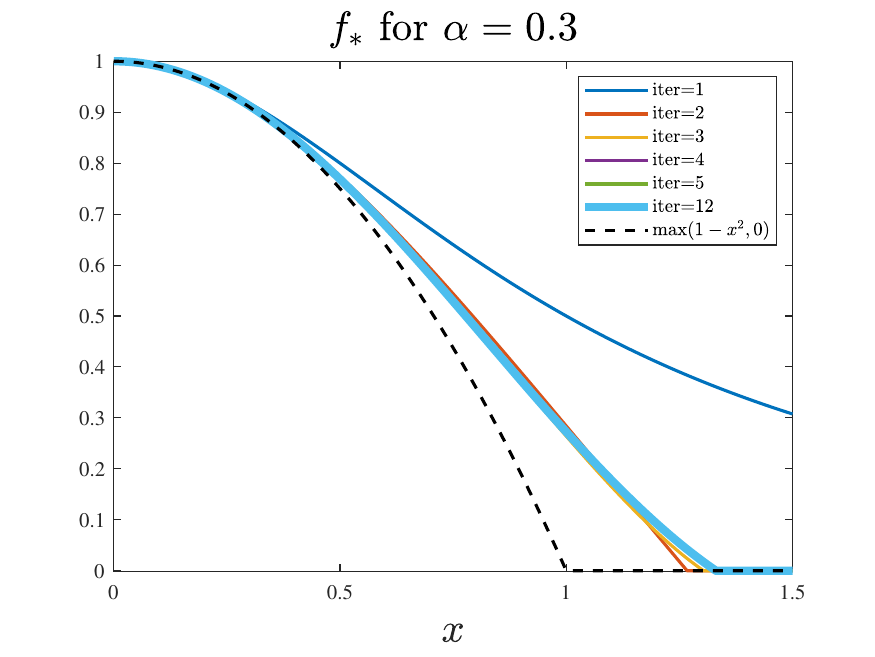}
\includegraphics[height=0.365\textwidth]{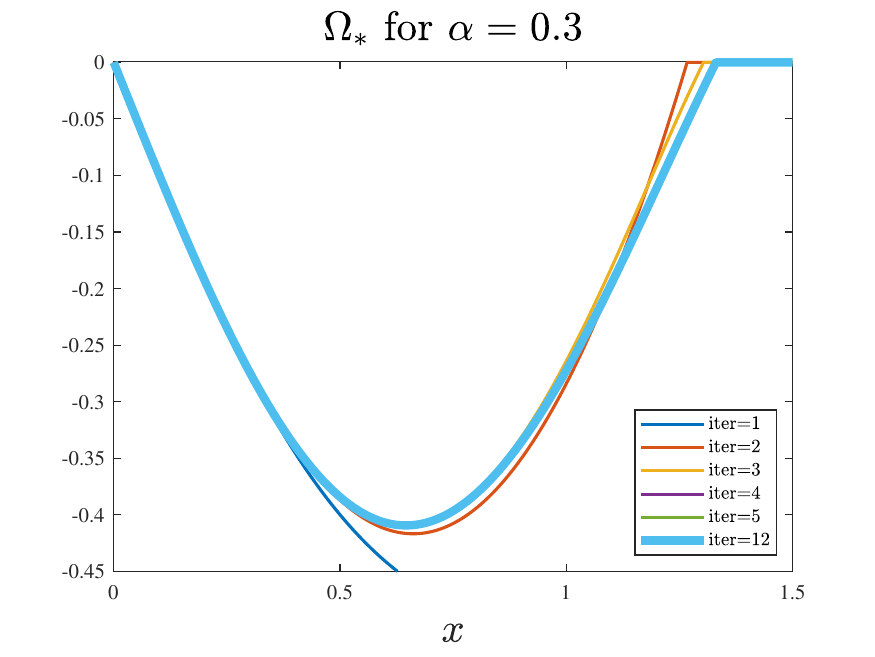}
\caption{Visualization for a self-similar profile for \eqref{Equation: fixedPointEqn}, the 1D reduction of gSQG on $\bbR^2$. Configuration:  $\alpha=0.3, N=5\times10^4$, domain truncation $[0,5]$, tolerance $\varepsilon=10^{-7}$, and the non-uniform mesh is $x_i=5(i/N)^2$. In the left subplot, we plot the iterations and the approximate profile for $f_*$, together with the lowerbound $\max(0,1-x^2)$. In the right subplot, we plot the iterations and the approximate profile for $\Omega_*=-xf_*$.  }
    \label{figure: 1DgSQGR2}
\end{figure}

Figure \ref{figure: 1DgSQGR2}  shows the functions during the iteration of $\scrR_\alpha$ for $\alpha=0.3$. In the plot, the colored curves are functions during iteration. The colored curve in bold is the profile at the end of iterations, which indicates the approximate fixed-point. The figures are consistent with our theoretical results that $f_*$ is monotone decreasing, compact support of $f_*$, and normalized scaling factors $\widetilde{c_\omega} = \widetilde{c_\ell } = -\frac{1}{2-2\alpha}$. We observed that the computed profiles and scaling parameters are stable under moderate changes of $N$.  For example, when we fix the domain truncation $[0,5]$ and test for $N=5\times10^4$ and $N=1\times10^4$, the residual error for $f_*$ is $3.06\times10^{-7}$. We do not need to change the domain truncation because the interval $[0,5]$ 
is sufficient to cover the compact support of $f_*$.

Moreover, we also study how approximate profiles change with respect to $\alpha\in(0,1)$. We use non-uniform mesh on $[0,5]$ with $N=1\times10^4$ grid points, specified by $x_i=5(i/N)^2$ for $i=0,1,\dots,N-1$. In other words, $\Delta x = 5\times10^{-8}$ at the origin and $\Delta x$ gradually increases to $1\times10^{-3}$ as $x$ approaches the endpoint $x=5$. Other setups are the same as in Figure \ref{figure: 1DgSQGR2}. We plot in Figure \ref{figure: 1DgSQGR2_varying_alpha}  self-similar profiles for different values of $\alpha\in(0,1)$. 

In particular, as $\alpha\to 0$ we formally have
$u=\calP_{0,\alpha}\omega\sim(-\Delta)^{\alpha-1}\omega \to (-\Delta)^{-1}\omega$. 
In this limiting case, the corresponding profile $\Omega$ to equation \eqref{Equation: SelfSimilarProfileOmega} should therefore be an eigenfunction of $(-\Delta)^{-1}$ on its support with Dirichlet boundary conditions.  Consequently, we expect
$f_*(x) \to \sin(\sqrt{6}\,x)/(\sqrt{6}\,x).$
Here the scaling factor $\sqrt{6}$ is chosen so that the normalization $f_*(x)\sim 1-x^2+O(x^4)$ holds near the origin. Although we do not establish the continuity of $f_*$ in $\alpha$, 
the limiting behavior is clearly observed numerically in Figure \ref{figure: 1DgSQGR2_alpha_001}, 
which also supports the accuracy of our computations.

\begin{figure}[!ht] 
\includegraphics[height=0.366\textwidth]{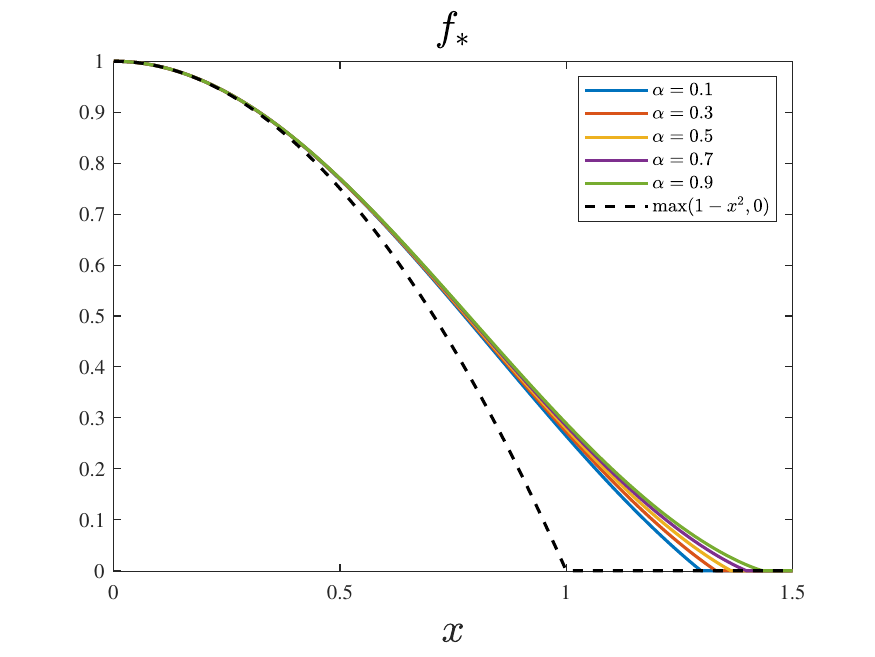}
\includegraphics[height=0.366\textwidth]{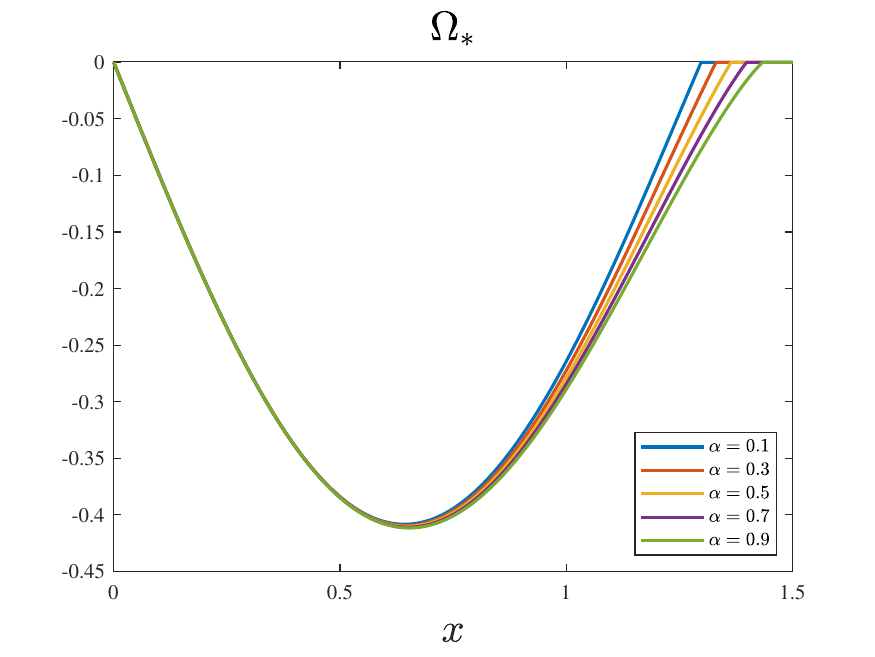}
\caption{Visualization for self-similar profiles for \eqref{Equation: fixedPointEqn}, the 1D reduction of gSQG on $\bbR^2$, under different values of $\alpha\in(0,1)$. Configuration:  $\alpha=0.1,0.3,0.5,0.7,0.9, N=1\times10^4$, domain truncation $[0,5]$, tolerance $\varepsilon=10^{-7}$, and the non-uniform mesh is $x_i=5(i/N)^2$. In the left subplot, we plot the approximate profiles $f_*$, together with the lowerbound $\max(0,1-x^2)$, for different $\alpha$. In the right subplot, we plot the iterations and the approximate profile $\Omega_*=-xf_*$ for different $\alpha$.  }
    \label{figure: 1DgSQGR2_varying_alpha}
\end{figure}

\begin{figure}[!ht] 
\includegraphics[height=0.366\textwidth]{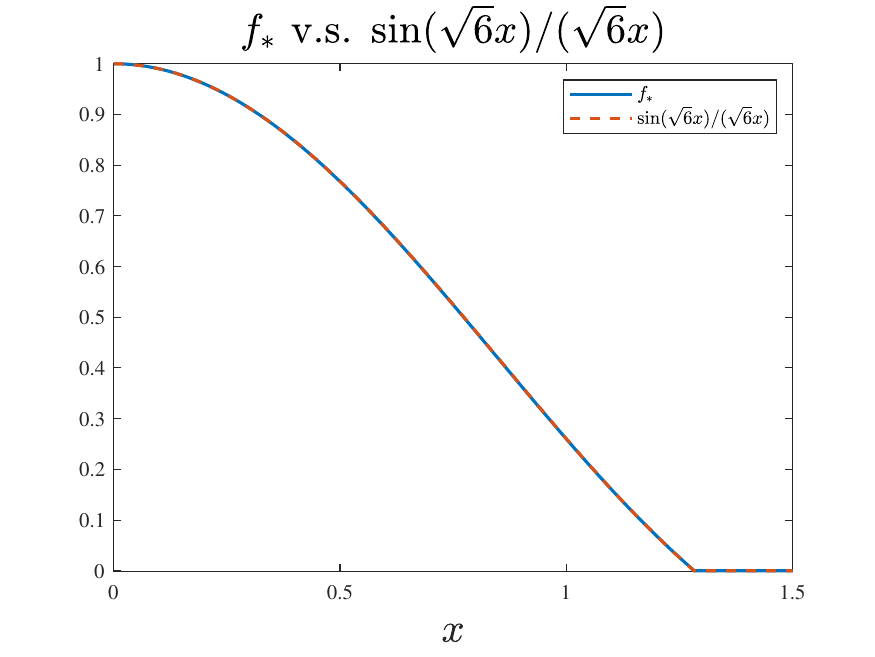}
\includegraphics[height=0.366\textwidth]{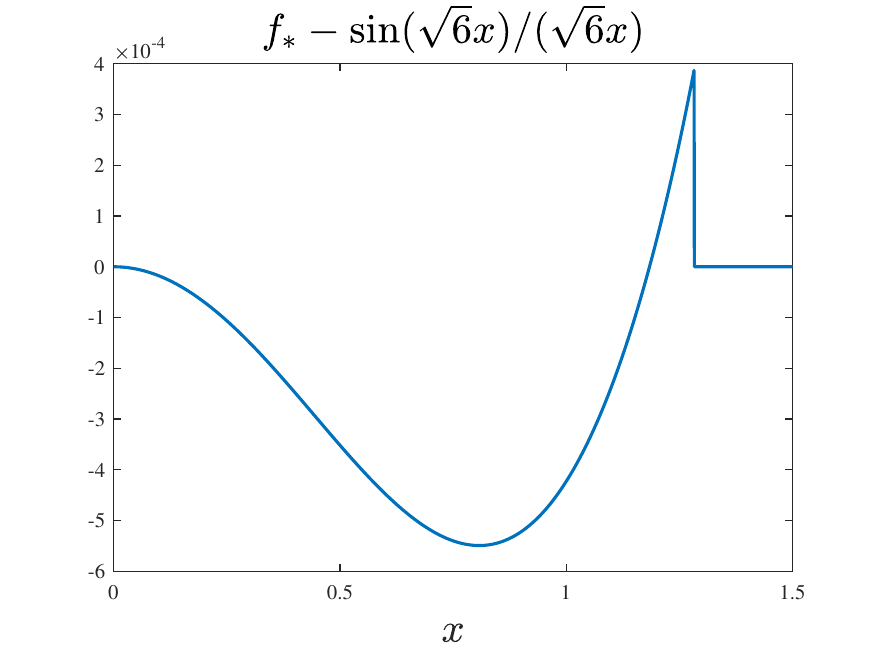}
\caption{Visualization for the limiting behavior of the self-similar profile of \eqref{Equation: fixedPointEqn} as $\alpha\to0$. 
Configuration: $\alpha=0.01$, $N=1\times10^4$, domain truncation $[0,5]$, tolerance $\varepsilon=10^{-7}$, and the non-uniform mesh is $x_i=5(i/N)^2$. 
In the left subplot, we plot the computed profile $f_*$ together with the predicted limiting profile $\sin(\sqrt{6}\,x)/(\sqrt{6}\,x)$. 
In the right subplot, we plot the difference $f_*(x)-\sin(\sqrt{6}\,x)/(\sqrt{6}\,x)$, showing that the numerical profile is already very close to the predicted limit when $\alpha$ is small. }
    \label{figure: 1DgSQGR2_alpha_001}
\end{figure}


\subsection{Numerical Simulation for 1D reduction of gSQG on $\HP$}
We use a nonuniform grid on $[0,2.43\times 10^{8}]$ with $N=5\times 10^4$ points, defined by $x_i=\sinh(20 i/N)$, $i=0,\dots,N-1$, to resolve both the near-origin region and the far field. The discrete operator provides an approximation of $U_x$; we then recover $U$ by a cumulative-sum procedure and compute the remaining quantities needed in the fixed-point map. We apply the iteration in \eqref{Equation: fixedPointEqnHP} to obtain a numerical approximation of a self-similar profile for \eqref{Equation: SelfSimilarProfileOmegaHP}. We take $\varepsilon=10^{-7}$ and stop when both the changes in $\Theta=xf$ and in the scaling parameter $c_\ell$ fall below $\varepsilon$.

\begin{figure}[!ht] 
\includegraphics[height=0.365\textwidth]{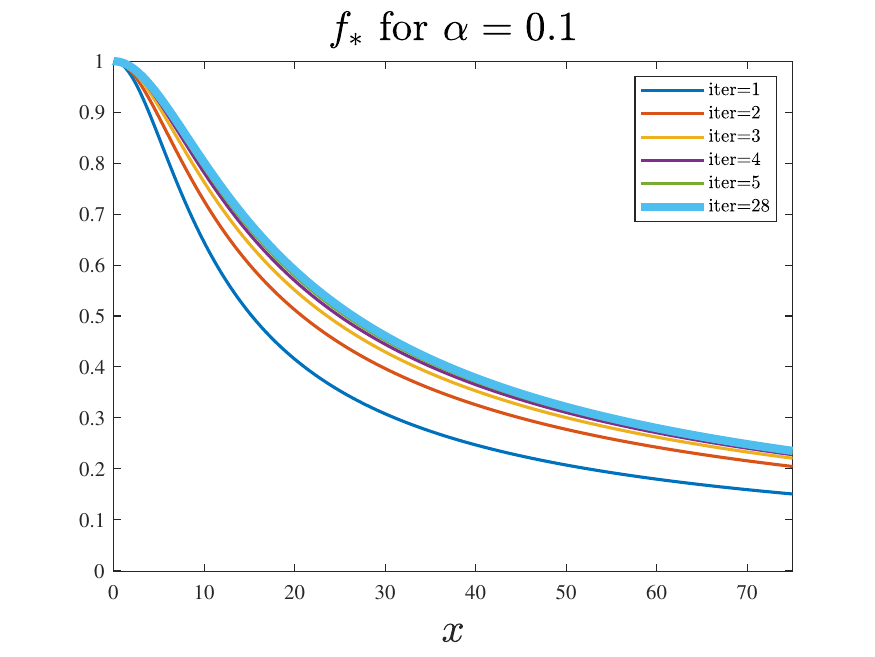}
\includegraphics[height=0.365\textwidth]{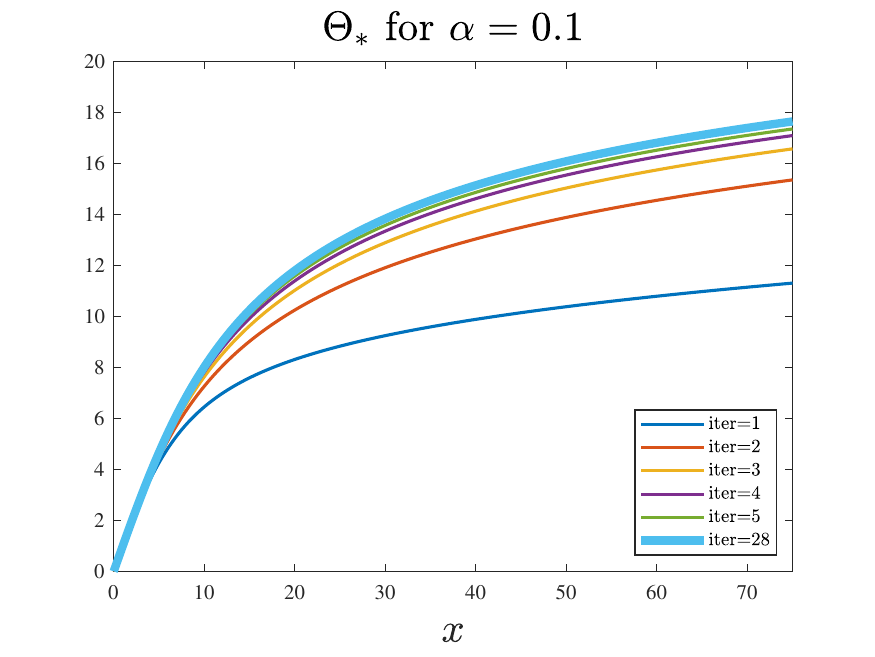}
\caption{Visualization for a self-similar profile for \eqref{Equation: fixedPointEqnHP}, the 1D reduction of gSQG on $\HP$. Configuration:  $\alpha=0.1, N=5\times10^4$, domain truncation $[0,2.43\times 10^{8}]$, tolerance $\varepsilon=10^{-7}$. We have $\Delta x\approx 4\times10^{-4}$ near origin and this $\Delta x$ increases as $x$ increases.  Then we compute and plot the profiles after scaling by the factor $\lambda$. In the left subplot,  we plot the iterations and the approximate profile for $f_*$.  In the right subplot, we plot the iterations and the approximate profile for $\Theta_*=xf_*$.  }
    \label{figure: 1DgSQGHP}
\end{figure}

Figure \ref{figure: 1DgSQGHP}  shows the functions during iteration of $\frR_\alpha$ for $\alpha=0.1$. In the plot, the colored curves are functions during the iteration. The bolded colored curve is the profile at the end of iterations, which indicates the approximate fixed-point. The figures are consistent with our theoretical results on that $f_*$ is monotone decreasing,  $f_*$ having no compact support, $c_\theta, c_\ell  >0$, $c_\theta-2\alpha c_\ell  = -1$. We observed that the computed profiles and scaling parameters are stable under moderate changes of $N$ and of the truncation range. For example, when we fix the domain truncation $[0,2.43\times 10^{8}]$ and test for $N=5\times10^4$ and $N=1\times10^4$,  the residual error for $c_\ell/c_\theta$ is $2.47\times10^{-5}$, and the residual error for $f_*$ is $2.86\times10^{-7}$. Moreover, if we fix $N=5\times10^4$ and test for domain truncation $[0,2.43\times 10^{8}]$ and $[0,4.87\times10^9]$, the residual error for $f_*$ is $2.74\times10^{-7}$, the residual error for $\widetilde{c_\ell}$ is $1.1\times10^{-2}$, and the residual error for $\widetilde{c_\theta}$ is $3.3\times10^{-3}$. This
indicates that the error from the truncation cannot be neglected.

Moreover, we also study how approximate profiles change with respect to $\alpha\in(0,1/2)$. We use a nonuniform grid on $[0,1.63\times 10^{6}]$ with $N=1\times 10^4$ points, defined by $x_i=\sinh(15 i/N)$, $i=0,\dots,N-1$, to resolve both the near-origin region and the far field. Other setups are the same as in Figure \ref{figure: 1DgSQGHP}. In Figure \ref{figure: 1DgSQGHP_varying_alpha} we plot self-similar profiles for different values of $\alpha\in(0,1/2)$. We note that we have rescaled the profile so that they all share $f_*(\sqrt{2}x/f_*''(0))\sim 1-x^2+O(x^4)$ when $x\to0$. This rescaling makes the shapes of the profiles comparable to each other. We note that rescaling each profile is equivalent to rescaling the corresponding parameters $c_\ell$ and $c_\theta$, but the ratio $c_\ell/c_\theta$ stays invariant under any rescaling.

Fortunately, in this scenario the limiting profile can also be analyzed.  Formally, as $\alpha\to1/2$, we have
$U\sim(-\Delta)^{\alpha-\frac12}\Theta \to \Theta$ in the profile equation \eqref{Equation: SelfSimilarProfileOmegaHP},
so the self-similar profile equation reduces to that of Burgers' equation. 
The latter admits the implicit solution
$\Theta_{\textbf{Burgers}}+\Theta_{\textbf{Burgers}}^3-x=0.$ 
Consequently, the limiting profile satisfies
$f_{\textbf{Burgers}}+x^2f_{\textbf{Burgers}}^3-1=0$
 as $\alpha\to 1/2.$
Although we do not establish the continuity of $f_*$ with respect to $\alpha$, this limiting behavior is observed in Figure \ref{figure: 1DgSQGHP_varying_alpha}, providing further evidence for the accuracy of the computations.

\begin{figure}[!ht] 
\includegraphics[height=0.365\textwidth]{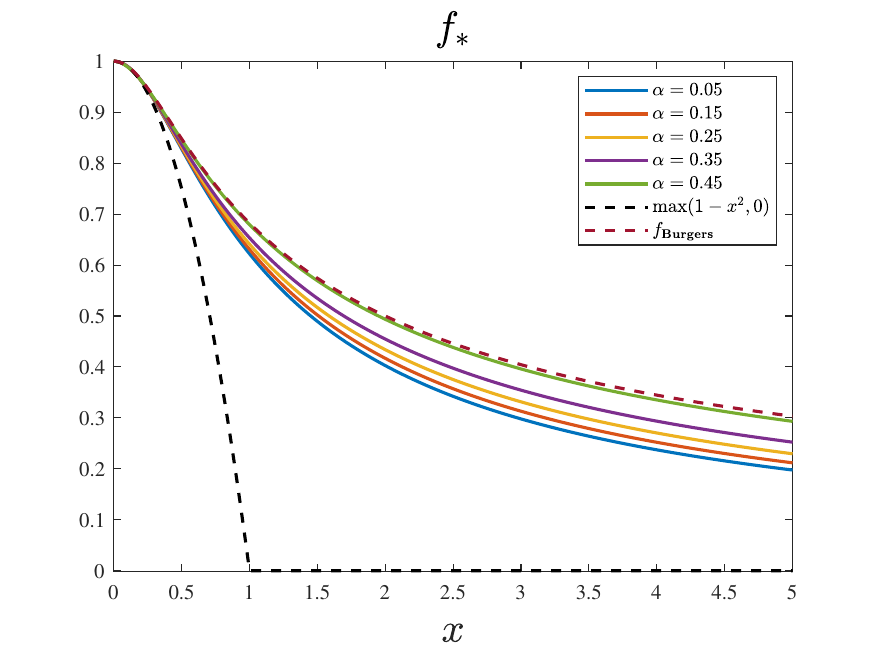}
\includegraphics[height=0.365\textwidth]{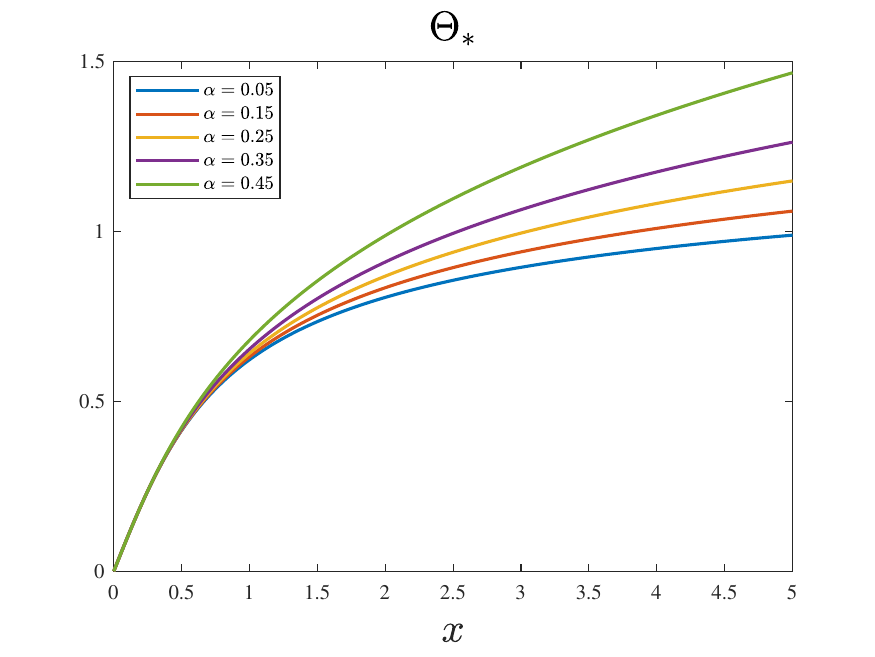}
\caption{Visualization for self-similar profiles for \eqref{Equation: fixedPointEqnHP}, the 1D reduction of gSQG on $\HP$, under different values of $\alpha\in(0,1/2)$. 
Configuration: $\alpha=0.05,0.15,0.25,0.35,0.45$, $N=1\times10^4$, domain truncation $[0,1.63\times 10^{6}]$, tolerance $\varepsilon=10^{-7}$. 
We have $\Delta x\approx 1.5\times10^{-3}$ near the origin and this $\Delta x$ increases as $x$ increases. 
In the left subplot, we plot the rescaled approximate profiles $f_*(\sqrt{2}x/f_*''(0))$, together with the lower bound $\max(0,1-x^2)$ and the Burgers limiting profile determined by $f_{\textbf{Burgers}}+x^2f_{\textbf{Burgers}}^3-1=0$, for different $\alpha$. 
In the right subplot, we plot the approximate profile $\Theta_*(\sqrt{2}x/f_*''(0))=\sqrt{2}x/f_*''(0)\cdot f_*(\sqrt{2}x/f_*''(0))$ . }
    \label{figure: 1DgSQGHP_varying_alpha}
\end{figure}

Finally, we also investigate how the parameters $c_\ell, c_{\theta}$ change with $\alpha\in(0,1/2)$.  We plot in Figure \ref{figure: 1DgSQGHP_varying_alpha_param} the value of ratio $c_\ell/c_{\theta}$ for different values of $\alpha\in(0,1/2)$. We also include the lowerbound $1/(2\alpha)$ from the proof and verify that we have indeed $c_\ell/c_\theta>1/(2\alpha)$.
\begin{figure}[!ht] 
\includegraphics[height=0.45\textwidth]
{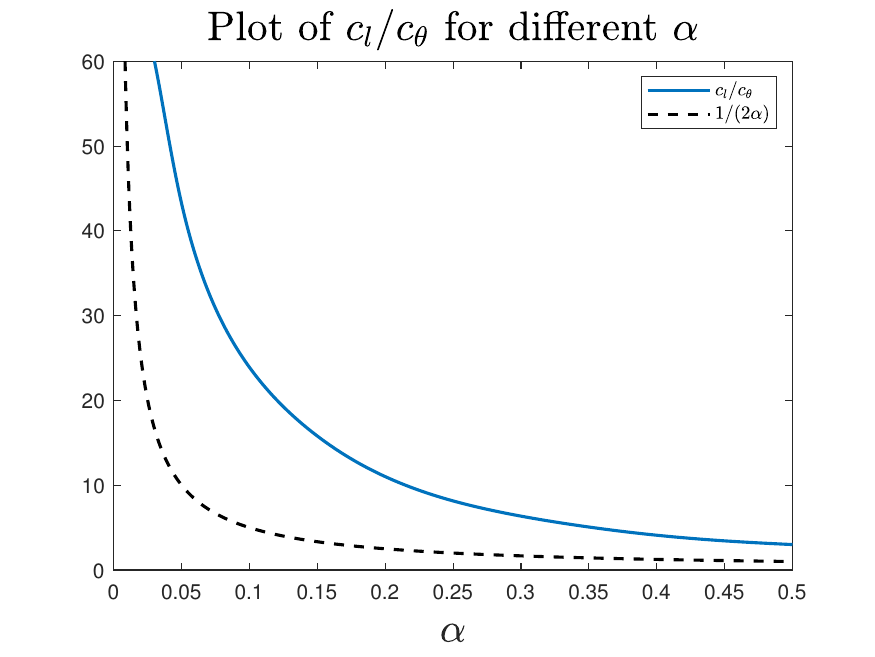}
\caption{Visualization for self-similar profiles for \eqref{Equation: fixedPointEqnHP}, the 1D reduction of gSQG on $\HP$, under different values of $\alpha\in(0,1/2)$. The setup is the same as in Figure \ref{figure: 1DgSQGHP_varying_alpha}. We plot value of $c_\ell/c_\theta$, together with the lowerbound $1/(2\alpha)$, for different $\alpha$.  }
    \label{figure: 1DgSQGHP_varying_alpha_param}
\end{figure}


\subsection{Numerical Simulation for gSQG on $\bbR^{2}_+$}
 We use a uniform grid on the square $[-16,16]^2$ with $N=2048$ grid points on $x$ and $y$ directions, i.e. $\Delta x = 1/64$. The numerical method is to approximate the singular convolution by discrete convolution. The advantage is that discrete convolution saves lots of RAM in computation (without building big matrix with size $\calO(N^2\times N^2)$) and it can be computed efficiently via the Fast Fourier Transform algorithm. The disadvantage is that using a uniform grid forces us to truncate out the tail and loses accuracy in the far field. Overall, the computation for the approximate self-similar profile for gSQG on $\HP$ is only for the purpose of illustrating the qualitative behavior of the profile.
 
 In Figures \ref{figure: gSQGHP_f}, \ref{figure: gSQGHP_-U1}, and \ref{figure: gSQGHP_U2}, the dashed curve indicates the initial function that we use to start the iteration. The colored curves are functions during the iteration. The bolded colored curve indicates the approximate fixed-point of the iteration, which means that this is an approximate self-similar profile to \eqref{Equation: gSQG-HP}. We observe that when $y=0$, the decay of $f_*$ is similar to that of the 1D reduction, meaning that our 1D reduction of gSQG is a meaningful model.

\begin{figure}[!ht] 
\includegraphics[width=\textwidth,page=1]{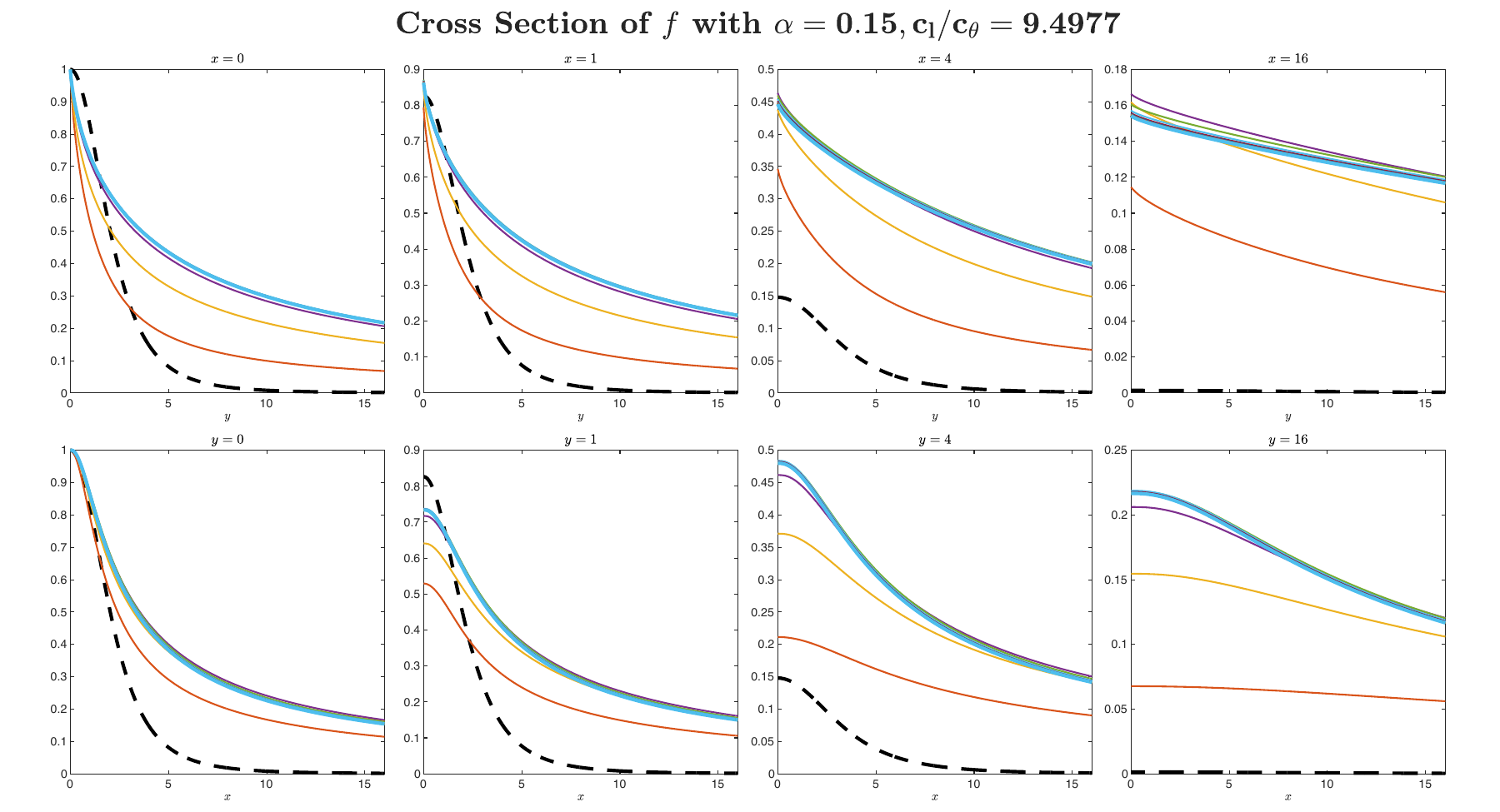}
\caption{Visualization for a candidate for the self-similar profile for \eqref{Equation: gSQG-HP}, the gSQG on $\HP$. Configuration:  $\alpha=0.15, N=2048,\Delta x=1/64$, domain truncation $[-16,16]^2$, tolerance $\varepsilon=10^{-7}$. The eight subplots are cross-sections of $f$ at $x=0,x=1,x=4, x=16, y=0, y=1, y=4, y=16$, respectively.  }
    \label{figure: gSQGHP_f}
\end{figure}

\begin{figure}[!ht] 
\includegraphics[width=\textwidth,page=2]{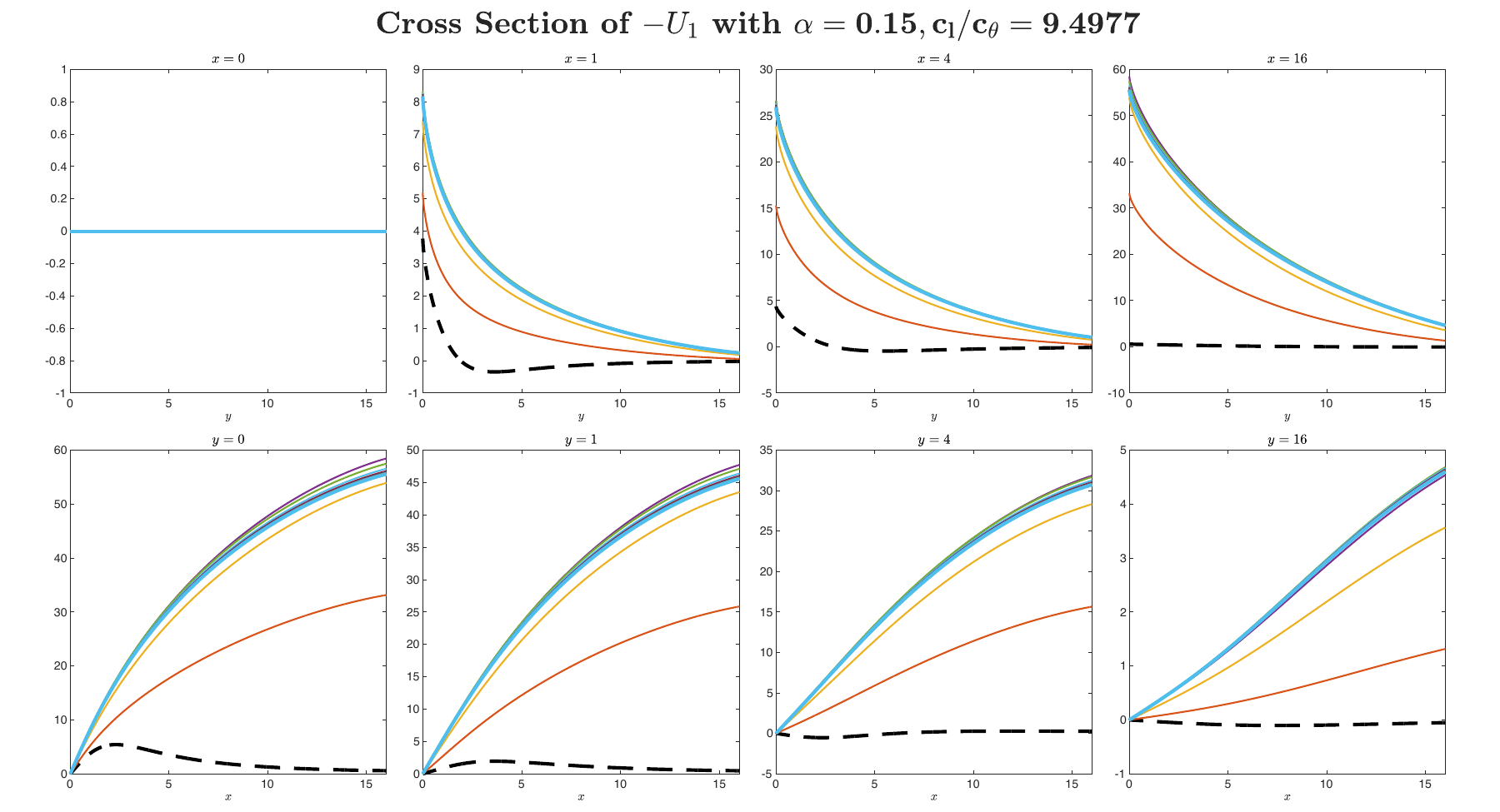}
\caption{Visualization for a candidate for the self-similar profile for \eqref{Equation: gSQG-HP}, the gSQG on $\HP$. Configuration:  $\alpha=0.15, N=2048,\Delta x=1/64$, domain truncation $[-16,16]^2$, tolerance $\varepsilon=10^{-7}$. The eight subplots are cross-sections of $-U1$ at $x=0,x=1,x=4, x=16, y=0, y=1, y=4, y=16$, respectively.  }
    \label{figure: gSQGHP_-U1}
\end{figure}

\begin{figure}[!ht] 
\includegraphics[width=\textwidth,page=3]{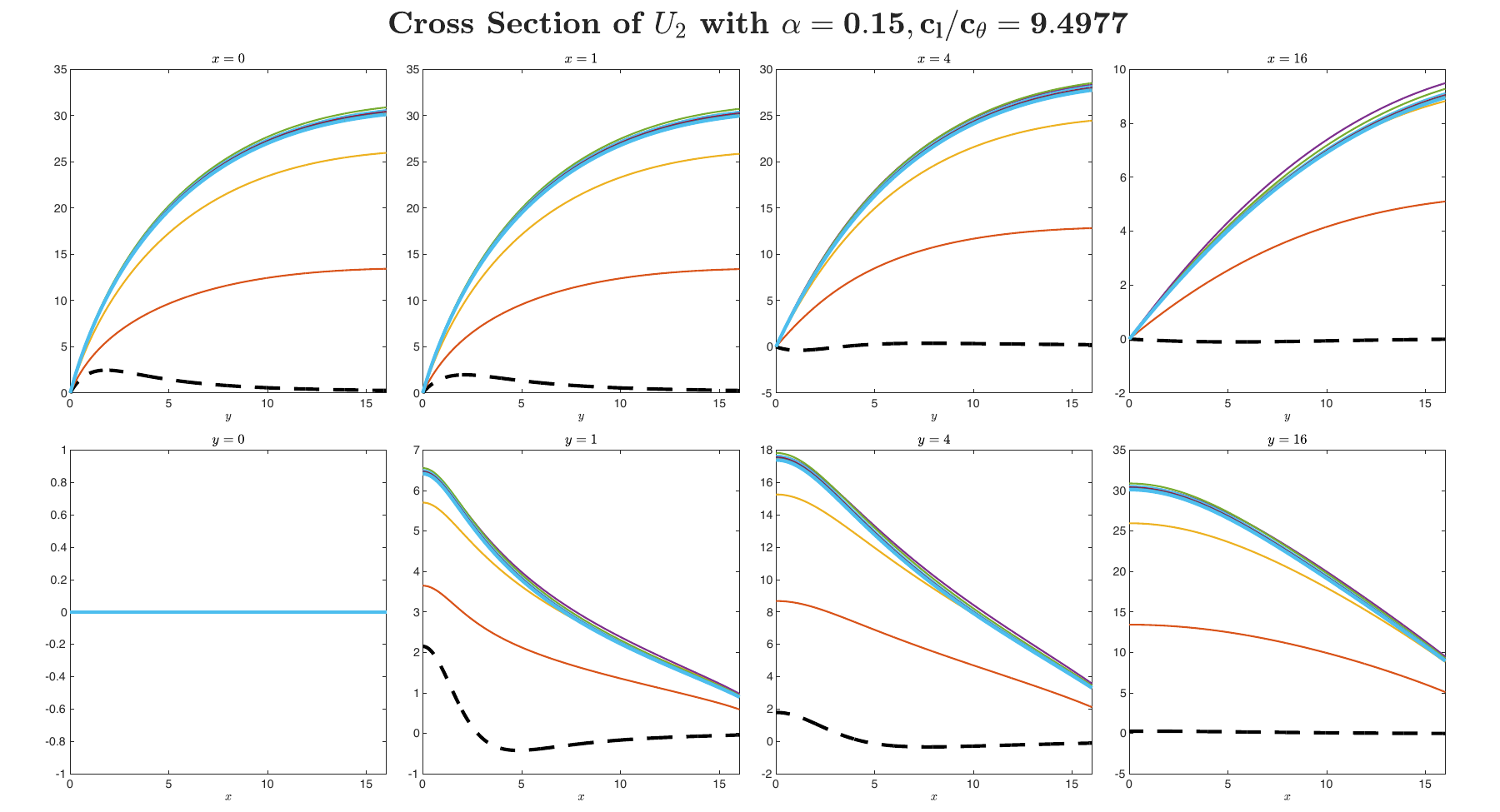}
\caption{Visualization for a candidate for the self-similar profile for \eqref{Equation: gSQG-HP}, the gSQG on $\HP$. Configuration:  $\alpha=0.15, N=2048,\Delta x=1/64$, domain truncation $[-16,16]^2$, tolerance $\varepsilon=10^{-7}$. The eight subplots are cross-sections of $U_2$ at $x=0,x=1,x=4, x=16, y=0, y=1, y=4, y=16$, respectively. }
    \label{figure: gSQGHP_U2}
\end{figure}


\section{Acknowledgments and Disclosure of Funding}
Prof. Thomas Y. Hou is supported by NSF grants DMS-2205590, DMS-2512878, the Choi Family Gift fund and Dr. Mike Yan Gift fund. Prof. Y. Sire is partially supported by DMS NSF grant 2154219, ``Regularity vs singularity formation in elliptic and parabolic equations''.

\appendix        
\section{Auxiliary functions}\label{Section: Auxiliary functions}

\begin{lemma}\label{Lemma: AuxiliaryFunctionF1}    
Let us define an auxiliary function $F_{1,\gamma}$ as follows: if $\gamma\in(-\infty,0)\cup(0,1)$, 
$$F_{1,\gamma}(t) := \frac{ 2\gamma (2+\gamma)t - (1+\gamma-t)(1+t)|1+t|^\gamma + (1-t)(1+\gamma+t)|1-t|^\gamma }{ \gamma (1+\gamma)(2+\gamma) t}, $$
and if $\gamma=0$,
$$F_{1,0}(t) := 1 - \frac{1-t^2}{2t}\log\abs{\frac{t+1}{t-1}} \ .$$
Then we have the following properties for $F_{1,\gamma}$ and $F_{1,\gamma}'$:
\begin{equation}
\begin{split}
&F_{1,\gamma}'(t) = 
\begin{cases}
 - \frac{(1+t^2)\rbracket{|1-t|^{\gamma} - |1+t|^{\gamma}  } + \gamma t \rbracket{ |1-t|^{\gamma} + |1+t|^{\gamma} } }{ \gamma (2+\gamma)t^2} & \text{ if } \gamma\neq0 \\
- \frac{1}{t} + \frac{t^2+1}{2t^2}\log\abs{\frac{t+1}{t-1}} & \text{ if } \gamma=0 \end{cases} \ , \\
&F_{1,\gamma}'(1/t) = t^{2-\gamma} F_{1,\gamma}'(t) \ , \\
&F_{1,\gamma}(0) =  F_{1,\gamma}'(0) = 0 \ ,  \\
& \lim_{t\to+\infty} F_{1,\gamma}(t) =\frac{2}{1+\gamma} \ , \\
&  \lim_{t\to0+} \frac{F_{1,\gamma}'(t)}{t} =\frac{2(2-\gamma)}{3} \ , \\
&  F_{1,\gamma}'(t) > 0 \text{ on } t\in(0,\infty)   \ .
  \end{split} 
\end{equation}
\end{lemma}
\begin{proof} When $\gamma\neq 0$, this can be proven using the Taylor expansion for $F_{1,\gamma}$ and $F_{1,\gamma}'$:
\begin{equation*}
\begin{split}
& F_{1,\gamma}(t) =  \sum_{k=1}^\infty \gamma^{-1}\begin{pmatrix} \gamma \\ 2k-1 \end{pmatrix} \frac{2k-\gamma}{k(2k+1)} t^{2k} \ , \\
& F_{1,\gamma}'(t) =  \sum_{k=1}^\infty \gamma^{-1}\begin{pmatrix} \gamma \\ 2k-1 \end{pmatrix} \frac{2(2k-\gamma)}{2k+1} t^{2k-1} \ , \\
& \gamma^{-1} \begin{pmatrix} \gamma \\ 2k-1 \end{pmatrix} = \frac{(\gamma-1)\cdots (\gamma-2k+2)}{(2k-1)!} > 0  \ .  \qedhere
\end{split} 
\end{equation*}
\end{proof}

\begin{lemma}\label{Lemma: Auxiliary Function2} 
We define the auxiliary function $F_{2,\gamma}$ by solving the following ODE: 
$$tF_{2,\gamma}'(t)- (1+\gamma) F_{2,\gamma}(t)  = t^{-1}F_{1,\gamma}'(t) - \frac{2(2-\gamma)}{3} \ , $$
which will give us $ -\partial_{\xi} (\xi^{1+\gamma} F_{2,\gamma}(x/\xi)) = \xi^{1+\gamma}/x F_{1,\gamma}'(x/\xi) - \frac{2(2-\gamma)}{3}\xi^\gamma $.
Then $F_{2,\gamma}$ has the following properties:
\begin{equation}
\begin{split}
& F_{2,\gamma}(t) := \frac{1}{3\gamma(1+\gamma)(2+\gamma)(4+\gamma)t^3} \Big[ 2\gamma(2-\gamma)(2+\gamma)(4+\gamma)t^3 \\
& \ \ \ \ \ \ \ \ \ \ \ + 3(1+\gamma+(2+\gamma-\gamma^2)t^2-3t^4)(|1-t|^{\gamma} - |1+t|^{\gamma} )  \\
&  \ \ \ \ \ \ \ \ \ \ \ +3((\gamma+\gamma^2)t-3\gamma t^3 )(|1-t|^{\gamma} + |1+t|^{\gamma})  \Big] \ , \\
& F_{2,\gamma}'(t) = - \frac{ (3+(2+\gamma^2)t^2+3t^4)(|1-t|^{\gamma}-|1+t|^{\gamma}) +3\gamma t(1+t^2)(|1-t|^{\gamma} + |1+t|^{\gamma}) }{\gamma(2+\gamma)(4+\gamma)t^4} \ , \\
& F_{2,\gamma}'(1/t) = t^{4-\gamma} F_{2,\gamma}'(t) \ , \\
& F_{2,\gamma}(0) = F_{2,\gamma}'(0)=0 \ , \\
&  \lim_{t\to\infty} F_{2,\gamma}(t) = \frac{2(2-\gamma)}{3(1+\gamma)} \ , \\
&  F_{2,\gamma}'(t)  > 0 \text{ on } t\in (0,\infty)  \ .
\end{split} 
\end{equation}  
\end{lemma}
\begin{proof}
We can apply the Taylor expansion to $F_{2,\gamma}$ and $F_{2,\gamma}'$:
\begin{equation*}
\begin{split}
& F_{2,\gamma}(t) =  \sum_{k=1}^\infty  \gamma^{-1} \begin{pmatrix} \gamma \\ 2k-1 \end{pmatrix} \frac{(2k-\gamma)(2k+2-\gamma)}{k(2k+1)(2k+3)} t^{2k} \ , \\
& F_{2,\gamma}'(t) =  \sum_{k=1}^\infty \gamma^{-1} \begin{pmatrix} \gamma \\ 2k-1 \end{pmatrix} \frac{2(2k-\gamma)(2k+2-\gamma)}{(2k+1)(2k+3)} t^{2k-1} \ , \\
& \gamma^{-1} \begin{pmatrix} \gamma \\ 2k-1 \end{pmatrix} = \frac{(\gamma-1)\cdots (\gamma-2k+2)}{(2k-1)!} > 0 \ . \qedhere
\end{split} 
\end{equation*}
\end{proof}

\begin{lemma}\label{Lemma:FracCommCutoff}
Let $\chi\in C_c^\infty(\bbR)$ and let $\gamma\ge 0$, then we have
\[
[(-\Delta)^{\gamma/2},\chi]: H^{t+\gamma-1}_{loc}(\bbR)\to H^t_{loc}(\bbR)
\]
continuously for all $t\in\bbR$.
\end{lemma}

\begin{proof}
This is a standard consequence of pseudodifferential commutator calculus and Sobolev mapping:
the high-frequency part of $(-\Delta)^{\gamma/2}$ is a classical pseudodifferential operator of order $\gamma$,
the low-frequency part is smoothing, and multiplication by $\chi$ has order $0$. Hence the commutator has order
$\gamma-1$, which yields the stated mapping by the Sobolev mapping theorem.
See \cite{taylor2006pseudo} for the commutator order drop
and for the Sobolev mapping theorem, together with the standard high/low frequency decomposition of the homogeneous symbol $|\xi|^\gamma$.
\end{proof}

\bibliographystyle{acm}
\bibliography{RefDatabase} 

\end{document}